\documentclass[11pt]{article}
\usepackage[english]{babel}
\usepackage{amsmath,amsthm}
\usepackage{amsfonts}
\usepackage{mathrsfs}
\usepackage{color}
\usepackage{bbm}
\usepackage{enumerate}
\usepackage{amsrefs}

\newtheorem{theorem}{Theorem}[section]
\newtheorem{corollary}[theorem]{Corollary}
\newtheorem{lemma}[theorem]{Lemma}
\newtheorem{proposition}[theorem]{Proposition}
\newtheorem{assumption}[theorem]{Assumption}
\theoremstyle{definition}
\newtheorem{definition}[theorem]{Definition}
\theoremstyle{remark}

\numberwithin{equation}{section}
\begin{document}
\def\Pro{{\mathbb{P}}}
\def\E{{\mathbb{E}}}
\def\e{{\varepsilon}}
\def\veps{{\varepsilon}}
\def\ds{{\displaystyle}}
\def\nat{{\mathbb{N}}}
\def\Dom{{\textnormal{Dom}}}
\def\dist{{\textnormal{dist}}}
\def\R{{\mathbb{R}}}
\def\O{{\mathcal{O}}}
\def\T{{\mathcal{T}}}
\def\Tr{{\textnormal{Tr}}}
\def\sgn{{\textnormal{sign}}}
\def\I{{\mathcal{I}}}
\def\A{{\mathcal{A}}}
\def\H{{\mathcal{H}}}
\def\S{{\mathcal{S}}}

\title{Existence and uniqueness for the mild solution of the stochastic heat equation with non-Lipschitz drift on an unbounded spatial domain %
\author{M. Salins}
\thanks{Accepted for publication in \textit{Stochastics and Partial Differential Equations: Analysis and Computations} August 31, 2020.}}%
\date{August 31, 2020}%
\maketitle
\begin{abstract}
  We prove the existence and uniqueness of the mild solution for a nonlinear stochastic heat equation defined on an unbounded spatial domain. The nonlinearity is not assumed to be globally, or even locally, Lipschitz continuous. Instead the nonlinearity is assumed to satisfy a one-sided Lipschitz condition.  First, a strengthened version of the Kolmogorov continuity theorem is introduced to prove that the stochastic convolutions of the fundamental solution of the heat equation and a spatially homogeneous noise grow no faster than polynomially. Second, a deterministic mapping that maps the stochastic convolution to the solution of the stochastic heat equation is proven to be Lipschitz continuous on polynomially weighted spaces of continuous functions. These two ingredients enable the formulation of a Picard iteration scheme to prove the existence and uniqueness of the mild solution.
\end{abstract}

\section{Introduction} \label{S:intro}
We prove the existence and uniqueness of a mild solution to the nonlinear stochastic heat equation
\begin{equation} \label{eq:SPDE}
  \begin{cases}
    \displaystyle{\frac{\partial u}{\partial t}(t,x) = \frac{1}{2}\Delta u(t,x) + f(u(t,x)) + \sigma(u(t,x))\dot{W}(t,x), \ \ t >0, \ \ x \in \mathbb{R}^d, }\\
    \displaystyle{u(0,x) = u_0(x)}
  \end{cases}
\end{equation}
in the case where $f:\mathbb{R} \to \mathbb{R}$ is not globally Lipschitz continuous. Instead we assume that there exists $\kappa \in \mathbb{R}$ such that for any $u_1< u_2$
\begin{equation} \label{eq:half-Lip}
  f(u_2) - f(u_1)\leq \kappa(u_2 - u_1),
\end{equation}
and $f$ satisfies a growth condition.

As a motivating example, let $f:\mathbb{R} \to \mathbb{R}$ be an odd-degree polynomial with negative leading coefficient
\[f(u) = -\alpha u^{2m+1} + \sum_{k=0}^{2m} a_k u^k\]
where $m \in \mathbb{N}$, $\alpha>0$ and $a_k \in \mathbb{R}$.  Such polynomials are not globally Lipschitz continuous, but they do satisfy \eqref{eq:half-Lip} where $\kappa = {\sup_u f'(u)}$, which is finite because $f'(u)$ is an even polynomial with negative leading term. Our set-up also allows us to consider some pathological drift terms that are not locally Lipschitz continuous such as decreasing functions that are not absolutely continuous with respect to Lebesgue measure.

In \eqref{eq:SPDE}, $\dot{W}$ is a Gaussian noise that is white in time and spatially homogeneous satisfying the strong Dalang condition (see Assumption \ref{assum:Dalang} below). We assume that $\sigma$ is globally Lipschitz continuous.

A mild solution to \eqref{eq:SPDE} is defined to be an adapted random field solution to the integral equation (see \cite{walsh})
\begin{align} \label{eq:mild}
  u(t,x) = &\int_{\mathbb{R}^d} G(t,x-y) u_0(y)dy + \int_0^t \int_{\mathbb{R}^d} G(t-s,x-y)f(u(s,y))dyds \nonumber\\
  &+ \int_0^t \int_{\mathbb{R}^d} G(t-s,x-y)\sigma(u(s,y))W(dyds),
\end{align}
where $G(t,x) := (2 \pi t)^{-\frac{d}{2}} e^{-\frac{|x|^2}{2t}}$ is the fundamental solution of the heat equation in $\mathbb{R}^d$.

The existence and uniqueness of solutions to \eqref{eq:SPDE} when $f$ and $\sigma$ are both globally Lipschitz continuous was proved by Dalang \cite{d-1999} using a Picard iteration argument that converges to a unique solution in the metric
\[\sup_{t \in [0,T]} \sup_{x \in \mathbb{R}^d} \E |u_{n+1}(t,x) - u_n(t,x)|^p\]
for appropriate choices of $T>0$ and $p>1$.
This metric will not be appropriate for proving existence and uniqueness in our setting.

The existence and uniqueness of stochastic reaction-diffusion equations defined on a bounded spatial domain $D \subset \mathbb{R}^d$ was proved by Cerrai \cite{c-2003} under the assumptions that the nonlinear reaction terms are locally Lipschitz continuous, have polynomial growth, and satisfy certain dissipativity properties. Because she assumed that the nonlinearities were locally Lipschitz continuous, she could construct a family $\{f_R(u)\}_{R>0}$ of globally Lipschitz continuous vector-fields that satisfy $f(u) = f_R(u)$ for all $|u|\leq R$. Using Picard iteration arguments in the metric
\[\E \left|\sup_{t \in [0,T]} \sup_{x \in D} |u_{n+1,R}(t,x) - u_{n,R}(t ,x)| \right|^p,\]
Cerrai proved that there exists a unique solution $u_R$ for any $R>0$ where the vector-field $f$ is replaced by $f_R$. Because the definition of $f_R$ guarantees that $f_R(u) = f(u)$ for $|u|\leq R$, the solutions $u_R(t,x)$ to the modified equation and the solution $u(t,x)$ to the original equation must coincide until the stopping time
\[\tau_R = \inf\left\{t>0: \sup_{x \in D} |u(t,x)| \geq R\right\},\]
for any $R>0$, where $D$ is the bounded spatial domain.

Then, using a-priori estimates derived from the dissipativity assumptions on $f$, Cerrai proved that these stopping times $\tau_R$ are finite for any $R>0$ and approach infinity as $R \to \infty$. Similar stopping time arguments have been used by many authors including Gy\"ongy and Rovira for Burgers-type equations \cite{gr-2000}. There have been other recent investigations of SPDEs on finite spatial domains and Banach-space-valued stochastic processes more generally that are exposed to nonlinearities that are not Lipschitz continuous \cite{dkz-2019,grt-2018,m-2013,mnq-2013,mr-2010}.

In the unbounded domain setting, such a stopping time argument does not work in general. If $\sigma$ is bounded away from zero then the random field solutions to \eqref{eq:SPDE} (if they even exist) will have the property that for any $t>0$
\[\Pro\left(\sup_{x \in \mathbb{R}^d} |u(t,x)| = +\infty \right) = 1.\]
The fact that the solution is unbounded with probability one prevents us from taking advantage of the local Lipschitz continuity of $f$.

Iwata published one of the earliest results proving existence and uniqueness for solutions to SPDEs defined on an unbounded space, like \eqref{eq:SPDE}, where the drift terms are not necessarily globally Lipschitz continuous \cite{i-1987}. Iwata assumed that the $f$ terms are locally Lipschitz continuous, have polynomial growth, and satisfy a dissipativity assumption similar to \eqref{eq:half-Lip}. Iwata's method compares the SPDE defined on the spatial domain $\mathbb{R}$ to a family of SPDEs defined on a bounded interval spatial domain and proving that the limit of the solutions, as the domains converge to the whole line, solves the desired SPDE.

When the noise term $\sigma \dot{W}$ is assumed to be particularly regular, then the mild solutions \eqref{eq:mild} may be bounded in space.  In this case, several authors have investigated the well-posedness of \eqref{eq:SPDE} with non-Lipschitz continuous drift.  Funaki studied a similar SPDE to \eqref{eq:SPDE}, with the added restriction that $\sigma(x) =0$ for $|x|\geq 1$ \cite{f-1995}.  Eckmann and Hairer investigated ergodic properties of an SPDE with a cubic nonlinearity defined on an unbounded spatial domain \cite{eh-2001}, but they made strong assumptions on the noise term that ensured that their solutions were bounded in space. Chow investigated a similar problem for semilinear stochastic wave equations defined on unbounded spatial domains \cite{c-2002}. 

Recently, the well-posedness for certain SPDEs defined on unbounded domains and exposed to additive space-time white noise has been proven. Bl\"omker and Han studied a stochastic complex Ginzburg-Landau equation exposed to additive space-time white noise in spatial dimension one \cite{bh-2010}. Bl\"omker and Han prove that the solutions exist in a weighted $L^2$ space and prove existence and uniqueness via finite-dimensional Galerkin approximations and a-priori bounds on the solutions.  Bianchi, Bl\"omker, and Schneider investigated proved the well-posedness for the Swift-Hohenberg equation on the whole line \cite{bbs-2017}, studying the problem in weighted function spaces and approximating the unbounded domain with finite domains. Both of these investigations considered the case of additive space-time white noise in spatial dimension one.

Mourrat and Weber recently proved the well posedness of the renormalized $\Phi^4$ model exposed two space-time white noise, defined on the two-dimensional plane \cite{mw-2017}. The result builds on the work of Jona-Lasinio and Mitter \cite{jm-1985}, Parisi and Wu \cite{pw-1981}, Hairer and Labb\'e \cite{hl-2015}, Da Prato and Debussche \cite{dd-2003}, and others. The well-posedness for \eqref{eq:SPDE} differs from the setting of \cite{mw-2017} because we assume that the noise satisfies the strong Dalang condition, prohibiting space-time white noise in dimension two. We do not require renormalization procedures, and the solutions in our setting are function-valued.

We also remind the reader of the existence and uniqueness results of  Mytnik, Perkins, and Sturm, and the non-uniqueness results of Mueller, Mytnik, and Perkins in the setting where when $\sigma$ fails to be globally Lipshchitz continuous \cite{mps-2006,mmp-2014,mp-2011}. In those papers, the authors considered the case where $f$ is Lipschitz continuous, but $\sigma$ is H\"older continuous. In the current paper we assume that $\sigma$ is globally Lipschitz continuous and that $f$ is not Lipschitz continuous. In the future it will be interesting to see if one can prove existence and uniqueness for \eqref{eq:SPDE} when $\sigma$ is only H\"older continuous and $f$ satisfies \eqref{eq:half-Lip}.

In this paper we develop a new method for proving existence and uniqueness of \eqref{eq:SPDE}.
First, we prove a version of the Kolmogorov continuity theorem. The classic Kolmogorov Theorem says that if there exists $x_0 \in \mathbb{R}^d$, $A>0$, $p>1$, and $\gamma\in (0,1)$ such that $\gamma p> d$ and a random field $X: \mathbb{R}^d \times \Omega \to \mathbb{R}$ satisfies
\[\E|X(x_0)|^p \leq A \text{ and } \E |X(x_1)-X(x_2)|^p \leq A |x_1 - x_2|^{\gamma p}, \text{ for } x_1,x_2 \in \mathbb{R}^d,\]
then there exists a H\"older continuous modification of this field (see Theorem 4.3 of \cite{dkmnx-2009}). In this paper we prove that such a process $X$ has the additional property that for any $\theta> \frac{\gamma p}{p-d}$,
\begin{equation}
  \E\left|\sup_{x \in \mathbb{R}^d} \frac{ |X(x)| }{1 + |x-x_0|^\theta} \right|^p \leq CA.
\end{equation}
We apply this theorem along with well-known estimates of Sanz-Sol\'e and Sarr\`a \cite{sss-1999} to prove that stochastic convolutions $Z(t,x)$  of the form
\begin{equation} \label{eq:stoch-conv-sigma}
  Z(t,x) = \int_0^t \int_{\mathbb{R}^d} G(t-s,x-y) \sigma(s,y)W(dyds)
\end{equation}
satisfy
\begin{equation} \label{eq:intro-stoch-conv-bound}
  \sup_{x_0 \in \mathbb{R}^d} \E \left|\sup_{t \in [0,T]} \sup_{x \in \mathbb{R}^d} \frac{Z(t,x)}{1 + |x-x_0|^\theta} \right|^p \leq C_{T,p,\theta}\sup_{t \in [0,T]}\sup_{x \in \mathbb{R}^d} \E|\sigma(t,x)|^p
\end{equation}
for appropriate choices of $p$ and $\theta$.

Then we consider an associated deterministic problem. Given a continuous deterministic function $z: [0,T] \times \mathbb{R}^d \to \mathbb{R}$ we try to find a solution to the integral equation.
\[\mathcal{M}(z)(t,x) = \int_0^t \int_{\mathbb{R}^d}G(t-s,x-y)f(\mathcal{M}(z)(s,y))dyds + z(t,x).\]
We prove that $\mathcal{M}$ is a well-defined function with nice properties when considered as a map $C_{x_0,\theta}([0,T]\times \mathbb{R}^d) \to C_{x_0,\theta}([0,T]\times \mathbb{R}^d)$ where these are weighted spaces of continuous functions $\psi:[0,T]\times \mathbb{R}^d \to \mathbb{R}$ that satisfy
\begin{equation}
  \lim_{|x| \to \infty} \sup_{t \in [0,T]} \frac{|\psi(t,x)|}{ 1 + |x-x_0|^{\theta}} = 0,
\end{equation}
endowed with the norm
\begin{equation}
  |\psi|_{C_{x_0,\theta}([0,T]\times\mathbb{R}^d)} := \sup_{t \in [0,T]}\sup_{x \in \mathbb{R}^d} \frac{|\psi(t,x)|}{1 + |x-x_0|^{\theta}}.
\end{equation}
Furthermore, we prove that the mapping $\mathcal{M}$ is Lipschitz continuous as a map from $C_{x_0,\theta}([0,T]\times \mathbb{R}^d)$ to itself, even though $f:\mathbb{R} \to \mathbb{R}$ fails to be Lipschitz continuous.

Once we prove that $\mathcal{M}$ is well-defined, we can observe that the mild solution to \eqref{eq:mild} will satisfy
\begin{equation}
  u(t,x) = \mathcal{M}(U_0 + Z)(t,x)
\end{equation}
where $U_0(t,x) = \int_{\mathbb{R}^d} G(t,x-y) u_0(y)dy$ and $Z$ is the stochastic convolution
\begin{equation} \label{eq:stoch-conv}
  Z(t,x) = \int_0^t \int_{\mathbb{R}^d} G(t-s,x-y) \sigma(u(s,y))W(dyds).
\end{equation}
We argue that proving the existence and uniqueness of the solution $Z$ to the equation
\begin{equation} \label{eq:Z-M-eq}
  Z(t,x)= \int_0^t \int_{\mathbb{R}^d} G(t-s,x-y)\sigma(\mathcal{M}(U_0 + Z)(s,y))W(dyds)
\end{equation}
is equivalent to proving the existence and uniqueness of $u$ solving \eqref{eq:mild}.

To prove the existence and uniqueness of $Z$ solving \eqref{eq:Z-M-eq}, we  build a Picard iteration scheme $Z_n$ that is a contraction in the metric
\[\sup_{x_0 \in \mathbb{R}^d} \E\left|\sup_{t \in [0,T]}\sup_{x \in \mathbb{R}^d} \frac{|Z_{n+1}(t,x) - Z_n(t,x)|}{1 + |x-x_0|^{\theta}}\right|^p.\]
The idea here is that while for any $t>0$, $Z(t,x)$ is almost surely unbounded in $x$, no particular $x$ value is likely to be big. We need to consider a weighted supremum norm inside of the expectation, but we can take the supremum over the center of the weight outside of the expectation. This approach is sort of a combination of the approaches used in \cite{c-2003} and \cite{d-1999}. We argue that the limit $Z$ of the $Z_n$ must solve \eqref{eq:Z-M-eq}. Finally, $u= \mathcal{M}(U_0 + Z)$ will be a mild solution to \eqref{eq:mild}.



In Section \ref{S:prelim}, we list our main assumptions and state the main existence and uniqueness result, Theorem \ref{thm:exist-unique}. In Section \ref{S:fnction-spaces}, we define the main function spaces that we will use in this paper including weighted spaces of continuous functions, the dual space of $C_0$ and the subdifferential, fractional Sobolev spaces, and H\"older spaces.  In Section \ref{S:fnction-spaces}, we also recall the fractional Sobolev embedding theorem. In Section \ref{S:Kolmogorov}, we prove that random fields satisfying the standard assumptions of the Kolmogorov continuity theorem belong to the weighted spaces of continuous function $C_{x_0,\theta}([0,T]\times \mathbb{R}^d)$ defined in Section \ref{S:fnction-spaces}. Then we apply this new version of the Kolmogorov Theorem to show that stochastic convolutions belong to these weighted spaces. In Section \ref{S:M}, we define the mapping $\mathcal{M}$ and prove many of its important properties including Lipschitz continuity and a priori bounds. In Section \ref{S:existence}, we prove Theorem \ref{thm:exist-unique}, which states that there exists a unique mild solution to \eqref{eq:mild}.

\section{Assumptions and main result} \label{S:prelim}
Fix a filtered probability space $(\Omega, \mathcal{F}, \mathcal{F}_t, \Pro)$.

\begin{assumption} \label{assum:Dalang}
  Let $\dot W$ be a spatially homogeneous Gaussian noise, adapted to the filtration $\mathcal{F}_t$, that is white in time and spatially correlated with correlation measure $\Lambda$. Formally,
  \begin{equation} \label{eq:covariance}
    \E[\dot W(t,x) \dot W(s,y)] = \delta(t-s) \Lambda(x-y).
  \end{equation}
  In the above expression, $\delta$ denotes the delta measure at zero and $\Lambda$ is positive and positive definite. Its Fourier transform $\mu = \mathscr{F}(\Lambda)$ is a positive measure satisfying the strong Dalang condition
  \begin{equation} \label{eq:Dalang-cond}
    \int_{\mathbb{R}^d} \frac{1}{1 + |\xi|^{2(1-\eta)}} \mu(d\xi)<+\infty
  \end{equation}
  for some $\eta \in (0,1)$.
\end{assumption}

\begin{assumption} \label{assum:sigma}
  The noise coefficient $\sigma: \mathbb{R} \to \mathbb{R}$ is globally Lipschitz continuous.
\end{assumption}

\begin{assumption} \label{assum:f}
  The drift $f:\mathbb{R} \to \mathbb{R}$ is continuous and has the property that there exists $\kappa\in \mathbb{R}$ such that for any $u_1,u_2 \in \mathbb{R}$ satisfying $u_1<u_2$,
  \begin{equation}
    \label{eq:f-cond-dissip}
    f(u_2) - f(u_1) \leq \kappa (u_2 - u_1).
  \end{equation}
  Additionally,  $f$ satisfies the growth condition that there exist constants $K>0$ and $\nu\geq 0$ such that for all $u \in \mathbb{R}$
  \begin{equation} \label{eq:f-growth}
    |f(u)| \leq  Ke^{K|u|^\nu}.
  \end{equation}
\end{assumption}
\noindent We comment that \eqref{eq:f-growth} is a much weaker assumption than the assumption that $f$ has polynomial growth.

The following proposition lists, without proof, two statements that are equivalent to \eqref{eq:f-cond-dissip}. In this proposition the function $\sgn: \mathbb{R} \to \mathbb{R}$ is defined to be
\begin{equation} \label{eq:sign-def}
  \sgn(u) := \begin{cases} -1 & \text{ if } u<0 \\ 0 & \text{ if } u=0 \\ 1 &\text{ if } u>0.\end{cases}
\end{equation}
\begin{proposition} \label{prop-f-cond-equiv}
  The following are equivalent to \eqref{eq:f-cond-dissip}
  \begin{enumerate}
    \item For all $u_1, u_2 \in \mathbb{R}$,
    \begin{equation} \label{eq:f-cond-sign}
      (f(u_1) - f(u_2))\sgn(u_1 - u_2) \leq \kappa|u_1 - u_2|.
    \end{equation}
    \item $\phi(u) := f(u) - \kappa u$ is a non-increasing function.
  \end{enumerate}
\end{proposition}

\begin{assumption} \label{assum:init-data}
  The initial data $u_0: \mathbb{R}^d \times \Omega \to \mathbb{R}$ is $\mathcal{F}_0$-measurable (independent of $\dot{W}$), and satisfies
  \begin{equation} \label{eq:init-data-assumption}
    \sup_{x_0 \in \mathbb{R}^d} \E \left|\sup_{x \in \mathbb{R}^d}\frac{|u_0(x)|}{1 + |x-x_0|^\theta}\right|^p< \infty.
  \end{equation}
  for some
  \begin{equation} \label{eq:p-cond}
    \theta \in (0,2/\nu) \ \ \text{ and } \ \  p> \max\left\{\frac{(1+\theta)(d+1)}{\theta}, \frac{2(d+1)}{\eta} \right\},
  \end{equation}
  where $\nu$ is the constant from \eqref{eq:f-growth}, $\eta$ is from the strong Dalang condition \eqref{eq:Dalang-cond}, and $d$ is the spatial dimension.

  If $u_0$ is deterministic, this assumption requires
  \begin{equation} \label{eq:u_0-bded}
    \sup_{x \in \mathbb{R}^d} |u_0(x)|< +\infty.
  \end{equation}
  In the case where $u_0$ is deterministic and bounded, \eqref{eq:init-data-assumption} holds for any $\theta$ and $p$ satisfying \eqref{eq:p-cond}.
\end{assumption}


The following is the main result of this paper.
\begin{theorem} \label{thm:exist-unique}
  Under Assumptions \ref{assum:Dalang}, \ref{assum:sigma}, \ref{assum:f}, and \ref{assum:init-data}, there exists a mild solution to \eqref{eq:mild}. Furthermore, this mild solution has the property that for  $p $ and $\theta$ satisfying \eqref{eq:init-data-assumption}-\eqref{eq:p-cond}, and any $T>0$,
  \begin{equation} \label{eq:solution-bounds}
    \sup_{x_0 \in \mathbb{R}^d} \E \left|\sup_{t \in [0,T]} \sup_{x \in \mathbb{R}^d} \frac{|u(t,x)|}{1 + |x-x_0|^\theta} \right|^p <+\infty.
  \end{equation}
  The solution is unique in the class of adapted random fields satisfying
  \begin{equation} \label{eq:Lp-bounded}
    \sup_{t \in [0,T]}\sup_{x \in \mathbb{R}^d} \E |\sigma(u(t,x))|^p<+\infty.
  \end{equation}
\end{theorem}

The proof of Theorem \ref{thm:exist-unique} is in Section \ref{S:existence}

\begin{corollary}
  Assume Assumptions \ref{assum:Dalang}, \ref{assum:sigma}, \ref{assum:f}, and \ref{assum:init-data}. If $\sigma$ is uniformly bounded, then there exists a unique solution to \eqref{eq:mild}.
\end{corollary}

\begin{proof}
  If $u$ is a mild solution to \eqref{eq:mild} and $\sigma$ is bounded, then \eqref{eq:Lp-bounded} is always satisfied. Therefore, by Theorem \ref{thm:exist-unique}, the solution to the mild equation is unique.
\end{proof}

\section{Function spaces} \label{S:fnction-spaces}
Let $C_b:=C_b(\mathbb{R}^d)$ be the space of bounded continuous functions $\varphi: \mathbb{R}^d \to \mathbb{R}$ endowed with the supremum norm
\[|\varphi|_{C_b} : = \sup_{x \in \mathbb{R}^d} |\varphi(x)|.\]
For any $T>0$, let $C_b([0,T]\times \mathbb{R}^d)$ be the space of bounded continuous functions $\psi: [0,T]\times\mathbb{R}^d \to \mathbb{R}$ endowed with the supremum norm
\[|\psi|_{C_b([0,T]\times\mathbb{R}^d)}:= \sup_{t \in [0,T]}\sup_{x \in \mathbb{R}^d} |\psi(t,x)|.\]
Let $C_0:=C_0(\mathbb{R}^d)$ be the closed subset of $\varphi \in C_b$ such that
\[\lim_{|x| \to \infty} |\varphi(x)|=0\]
and for $T>0$ let $C_0([0,T]\times\mathbb{R}^d)$ be the closed subset of $\psi \in C_b([0,T]\times\mathbb{R}^d)$ such that
\[\lim_{|x| \to \infty}\sup_{t\in[0,T]} |\psi(t,x)| = 0.\]
These are all Banach spaces (see page 65 of \cite{conway}).

For $x_0 \in \mathbb{R}^d$ and $\theta\geq 0$ define the weighted space $C_{x_0,\theta}(\mathbb{R}^d)$ of continuous functions $\varphi: \mathbb{R}^d \to \mathbb{R}$ such that
\[\lim_{|x| \to \infty} \frac{ |\varphi(x)|}{1+|x-x_0|^\theta} = 0\]
and endow this space with the norm
\begin{equation*}
  |\varphi|_{C_{\theta,x_0}} := \sup_{x \in \mathbb{R}^d} \frac{ |\varphi(x)|}{1 + |x-x_0|^\theta}.
\end{equation*}
Similarly, for $T>0$, $x_0 \in \mathbb{R}^d$, and $\theta\geq 0$ we define $C_{x_0,\theta}([0,T]\times\mathbb{R}^d)$ to be the space of continuous functions $\psi:[0,T]\times \mathbb{R}^d \to \mathbb{R}$ such
\[\lim_{|x| \to \infty}\sup_{t \in [0,T]} \frac{|\psi(t,x)|}{1 + |x-x_0|^\theta} =0\]
and endow this space with the norm
\[|\psi|_{C_{x_0,\theta}([0,T]\times\mathbb{R}^d)} := \sup_{t \in [0,T]}\sup_{x \in \mathbb{R}^d} \frac{|\psi(t,x)|}{1 + |x-x_0|^\theta}.\]
We will call the point $x_0 \in \mathbb{R}^d$ in the above definitions the center of the weight.

These weighted spaces $C_{x_0,\theta}(\mathbb{R}^d)$ and $C_{x_0,\theta}([0,T]\times \mathbb{R}^d)$ are Banach spaces. Furthermore, for fixed $\theta>0$, the $C_{x_0,\theta}(\mathbb{R}^d)$ spaces coincide over all $x_0 \in \mathbb{R}^d$ and the $C_{x_0,\theta}([0,T]\times \mathbb{R}^d)$ coincide over all $x_0 \in \mathbb{R}^d$ because the $|\cdot|_{C_{x_0,\theta}}$ and $|\cdot|_{C_{x_1,\theta}}$ norms are equivalent for any fixed $x_0,x_1 \in \mathbb{R}^d$. On the other hand, it will not be sufficient to fix an $x_0$ (say $x_0=0$) and always use the $|\cdot|_{C_{0,\theta}}$ norm. Our results require that many estimates are uniform with respect to the center of the weights. For this reason it is convenient to denote the $x_0$ in the definition of the weighted spaces and their norms.

\subsection{The dual space of $C_0$ and subdifferentials}
Let $C_0^\star$ denote the dual space of $C_0=C_0(\mathbb{R}^d)$. The Riesz representation for this space is well-known.
\begin{proposition}[Theorem C.18 of \cite{conway}]
  The dual space of $C_0(\mathbb{R}^d)$, $C_0^\star(\mathbb{R}^d)$, is isomorphic to the set of signed regular finite Borel measures under the duality
  \begin{equation}
    \left<\varphi, \mu \right>_{C_0,C_0^\star} = \int_{\mathbb{R}^d} \varphi(x)\mu(dx).
  \end{equation}
  The norm of the dual space is the total variation norm,
  \[|\mu|_{C_0^\star}: = \sup_{|\varphi|_{C_0} \leq 1} \left< \varphi, \mu \right>_{C_0,C_0^\star}.\]

\end{proposition}

\begin{definition}[Subdifferential]
For any $v \in C_0(\mathbb{R}^d)$, the subdifferential is defined to be the subset of the dual space (see \cite[Appendix D]{dpz})
  \begin{equation} \label{eq:subdiff}
    \partial|v|_{C_0}: = \left\{\mu \in C_0^\star: \left<v,\mu\right>_{C_0,C_0^\star} = |v|_{C_0}, |\mu|_{C_0^\star} = 1 \right\}.
  \end{equation}
\end{definition}
\begin{proposition} \label{prop:subdff}
If $\mu \in \partial |v|_{C_0}$, then $\mu$ is a signed measure of the form $\mu(dx) = \sgn(v(x))\tilde{\mu}(dx)$ where $\tilde{\mu}$ is a positive measure supported on the set $\{x \in \mathbb{R}^d: |v(x)| = |v|_{C_0}\}$ and satisfying $\tilde{\mu}(\mathbb{R}^d) = 1$.
\end{proposition}

\begin{proof}
  If $\mu$ is of the form described above then
  \[|\mu|_{C_0^\star} = 1\] and
  \[\left< v, \mu \right>_{C_0,C_0^\star} = \int_{\mathbb{R}^d} v(x)\sgn(v(x)) \tilde{\mu}(dx) = \int_{\mathbb{R}^d} |v(x)|\tilde{\mu}(dx) = |v|_{C_0}.\]

  On the other hand, if $|\mu|_{C_0^\star}=1$ is not of the form described above, then
  \[\left<v,\mu\right>_{C_0,C_0^\star}< |v|_{C_0}.\]
\end{proof}

The subdifferential can be used to bound the upper-left-derivative  of a function of space and time.
\begin{definition}
  For a real-valued function $\varphi: \mathbb{R} \to \mathbb{R}$, the upper-left-derivative is defined to be
  \begin{equation} \label{eq:left-deriv}
    \frac{d^-}{dt} \varphi(t) = \limsup_{h \downarrow 0} \frac{\varphi(t) - \varphi(t-h)}{h}.
  \end{equation}
\end{definition}

The next proposition is similar to Proposition D.4 of \cite{dpz}, but does not require that the mapping $t \mapsto v(t,\cdot)$ be differentiable in the $C_0$ norm. We do not expect functions $v:[0,T]\times\mathbb{R}^d \to \mathbb{R}$ that have continuous first derivatives to be continuous in such a strong sense.
\begin{proposition} \label{prop-left-deriv-subdiff}
Let $v:[0,T]\times \mathbb{R}^d \to \mathbb{R}$ be a function that has continuous first derivatives and satisfies $v(t,\cdot) \in C_0$ for any $t \in [0,T]$. For any $t \in (0,T]$ and $\mu \in \partial|v(t,\cdot)|_{C_0}$ defined in \eqref{eq:subdiff}, the upper-left-derivatives of the mapping $t \mapsto |v(t,\cdot)|_{C_0}$ can be bounded by
\[\frac{d^-}{dt}|v(t,\cdot)|_{C_0} \leq \left<\frac{\partial v}{\partial t}(t,\cdot), \mu \right>_{C_0,C_0^\star}.\]
\end{proposition}

\begin{proof}
  Fix $t_0 \in (0,T]$. Let $x_0 \in \mathbb{R}^d$ be such that $x_0$ is a global maximizer or minimizer for $v(t_0,\cdot)$ in the sense that
  \[|v(t_0,x_0)| = v(t_0,x_0)\sgn(v(t_0,x_0)) = |v(t_0,\cdot)|_{C_0}.\]
  Observe that because $|\cdot|_{C_0}$ is a supremum norm, for any $h \in (0,t_0)$, we have the simple inequality
  \[|v(t_0-h,\cdot)|_{C_0}\geq v(t_0-h,x_0)\sgn(v(t_0,x_0)).\]
  Therefore,
  \begin{align} \label{eq:left-deriv-at-x_0}
    &\frac{d^-}{dt}|v(t,\cdot)|_{C_0} \Bigg|_{t = t_0} \nonumber \\
    &= \limsup_{h \downarrow 0} \frac{|v(t_0,\cdot)|_{C_0} - |v(t_0-h,\cdot)|_{C_0}}{h} \nonumber\\
    &\leq \limsup_{h \downarrow 0} \left(\frac{v(t_0,x_0) - v(t_0-h,x_0)}{h}\right)\sgn(v(t_0,x_0))\nonumber\\
    & \leq \frac{\partial v}{\partial t} (t_0,x_0)\sgn(v(t_0,x_0)).
  \end{align}
  If $\mu \in \partial |v(t_0,\cdot)|_{C_0}$, then by Proposition \ref{prop:subdff} $\mu(dx)= \sgn(v(t_0,x))\tilde{\mu}(dx)$ where $\tilde{\mu}$ is a positive unit measure supported on the set of $x_0$ for which \eqref{eq:left-deriv-at-x_0} holds. Then integrating both sides of \eqref{eq:left-deriv-at-x_0} against $\tilde{\mu}$ we conclude that
  \begin{equation}
    \frac{d^-}{dt} |v(t,\cdot)| \leq \left<\frac{\partial v}{\partial t}(t,\cdot), \mu \right>_{C_0,C_0^\star}.
  \end{equation}
\end{proof}

\subsection{Fractional Sobolev spaces}
  Given a bounded, open domain $D \subset \mathbb{R}^d$ with smooth boundary, $s \in (0,1)$ and $p\geq 1$, the fractional Sobolev space $W^{s,p}(D)$ is defined to be the Banach space of functions on $D$ endowed with  the norm
  \begin{equation} \label{eq:frac-Sobolev-def}
    |\varphi|_{W^{s,p}(D)}^p:= \int_D |\varphi(x)|^p dx + \int_D \int_D \frac{|\varphi(x) - \varphi(y)|^p}{|x-y|^{d+ sp}}dxdy.
  \end{equation}

  The H\"older spaces $C^\theta(D)$, $\theta \in (0,1)$ are the Banach spaces of functions on $D$ endowed with the norm
  \begin{equation} \label{eq:Holder-norm-def}
    |\varphi|_{C^\theta(D)} := \sup_{x \in D} |\varphi(x)| + \sup_{\substack{x,y \in D\\x \not= y }} \frac{|\varphi(x) - \varphi(y)|}{|x-y|^\theta}.
  \end{equation}

  \begin{proposition}[Fractional Sobolev embedding theorem (Theorem 8.2 of \cite{sobolev})] \label{prop:Sobolev-embedding}
    If $s \in (0,1)$ and $sp>d$, then $W^{s,p}(D)$ embeds continuously into $C^\theta(D)$ where $\theta= \frac{sp-d}{p}$ and there exists $C>0$ such that for all $\varphi \in W^{s,p}(D)$,
    \begin{equation} \label{eq:Sobolev-embedding}
      |\varphi|_{C^\theta(D)} \leq C |\varphi|_{W^{s,p}(D)}.
    \end{equation}
  \end{proposition}


\section{The Kolmogorov Theorem on unbounded domains} \label{S:Kolmogorov}
\begin{theorem} \label{thm:new_Kolmo}
  Let $(\Omega, \mathcal{F}, \Pro)$ be a probability space. Let $d \in \mathbb{N}$, $p>d$ and $\gamma \in (0,1)$ such that $p\gamma>d$. Let $\theta> \frac{p\gamma}{p-d}$. There exists $C_{p,d,\gamma,\theta}>0$ such that
  for any  random field $X: \mathbb{R}^d \times \Omega \to \mathbb{R}$ satisfying for some $A>0$ and $x_0 \in \mathbb{R}^d$
  \begin{equation} \label{eq:Kolmo-assum-x0}
    \E|X(x_0)|^p \leq A
  \end{equation}
  and for all $x_1, x_2 \in \mathbb{R}^d$
  \begin{equation} \label{eq:Kolmo-assum-diff}
    \E|X(x_1) - X(x_2)|^p\leq A|x_1 - x_2|^{p \gamma},
  \end{equation}
  this random field is almost surely continuous and satisfies
  \begin{equation} \label{eq:Kolmogorov}
    \E\left( \sup_{x \in \mathbb{R}^d} \frac{|X(x)|}{1 + |x-x_0|^{\theta}} \right)^p \leq C_{p,d,\gamma,\theta} A.
  \end{equation}
  Furthermore,
  \begin{equation} \label{eq:X-in-C-theta}
    \Pro\left(\lim_{|x| \to \infty} \frac{|X(x)|}{1 + |x-x_0|^\theta} = 0 \right)=1.
  \end{equation}
  Therefore, $X \in C_{x_0,\theta}(\mathbb{R}^d)$ with probability one.
\end{theorem}

\begin{proof}
  The proof is based on a Kolmogorov continuity theorem result, which is a consequence of fractional Sobolev embeddings.

  Let $\theta> \frac{p \gamma}{p-d}$. Let $B(0,1) = \{z \in \mathbb{R}^d: |z|< 1\}$ be the open unit ball and $\overline{B(0,1)}$ denote the closed unit ball. Define the random field $Y:\overline{B(0,1)}\times \Omega \to \mathbb{R}$ by
  \begin{equation}\label{eq:Y-def}
    Y(z) =
    \begin{cases}
      (1-|z|) X \left(x_0 + z(1-|z|)^{-\frac{1}{\theta}} \right) & \text{ if } |z|<1\\
      0 & \text{ if } |z|=1.
    \end{cases}
  \end{equation}

  In the following, the constant $C$ may change from line to line, but it only depends on $\gamma$, $p$, $d$, and $\theta$.

  Notice that the function $\rho: {B(0,1)} \to \mathbb{R}^d$ given by $$\rho(z) := x_0 + z(1-|z|)^{-\frac{1}{\theta}}$$
  maps $B(0,1)$ one-to-one and onto $\mathbb{R}^d$. You can see this by observing that the real-valued function $t \mapsto t(1 - t)^{-\frac{1}{\theta}}$ is strictly increasing and maps $[0,1)$ onto $[0,+\infty)$.



  We show that the random field $Y$ defined in \eqref{eq:Y-def} satisfies a Kolmogorov continuity criterion.

  If $z_1, z_2 \in B(0,1)$, then
  \begin{align*}
    \E&|Y(z_1) - Y(z_2)|^p\\
    \leq &C |(1 - |z_1|) - (1-|z_2|)|^p \E |X(x_0 + z_1(1-|z_1|)^{-\frac{1}{\theta}})|^p\\
    &+ C (1- |z_2|)^p \E |X(x_0 + z_1(1-|z_1|)^{-\frac{1}{\theta}}) - X(x_0 + z_2(1-|z_1|)^{-\frac{1}{\theta}})|^p\\
    &+ C(1- |z_2|)^p \E |X(x_0 + z_2(1-|z_1|)^{-\frac{1}{\theta}})- X(x_0 + z_2(1-|z_2|)^{-\frac{1}{\theta}}) |^p\\
    =:& I_1 + I_2 +I_3.
  \end{align*}
  Without loss of generality, we can assume $1-|z_2| \leq 1 - |z_1|$.

  Using the fact that $\frac{p\gamma}{\theta}< p-d<p$,
  \[|(1 - |z_1|)- (1 - |z_2|)|^p \leq (2(1-|z_1|))^{\frac{p\gamma}{\theta}} ||z_1| - |z_2||^{p-\frac{ p\gamma}{\theta}}, \]
  and we see that
  \begin{equation*}
    I_1 \leq C(1 - |z_1|)^{\frac{p\gamma}{\theta} }|z_1-z_2|^{p-\frac{p\gamma}{\theta}}\E|X(x_0 + z_1(1-|z_1|)^{-\frac{1}{\theta}})|^p.
  \end{equation*}
  We can bound the expectation by
  \begin{align*}
    &\E|X(x_0 + z_1(1-|z_1|)^{-\frac{1}{\theta}})|^p \\
    &\leq C\E|X(x_0 + z_1(1-|z_1|)^{-\frac{1}{\theta}})-X(x_0)|^p + C\E|X(x_0)|^p.
  \end{align*}
   It follows from \eqref{eq:Kolmo-assum-x0} and \eqref{eq:Kolmo-assum-diff} that
  \[\E|X(x_0 + z_1(1-|z_1|)^{-\frac{1}{\theta}})|^p \leq CA (1 + |z_1|^{p \gamma} (1 - |z_1|)^{-\frac{p\gamma}{\theta}}).\]
  Therefore,
  \begin{equation}\label{eq:I_1-Kolmo}
   I_1 \leq CA (1+(1 - |z_1|)^{\frac{p\gamma}{\theta} - \frac{p\gamma}{\theta}}) |z_1-z_2|^{p-\frac{p\gamma}{\theta}} \leq CA |z_1 - z_2|^{p - \frac{p\gamma}{\theta}}.
  \end{equation}

  We now estimate $I_2$. By \eqref{eq:Kolmo-assum-diff},
  \[\E |X(x_0 + z_1(1-|z_1|)^{-\frac{1}{\theta}}) - X(x_0 + z_2(1-|z_1|)^{-\frac{1}{\theta}})|^p \leq CA |z_1-z_2|^{\gamma p} (1-|z_1|)^{-\frac{p \gamma}{\theta}}.\]
  We can use the assumption that
   $(1-|z_2|)\leq (1-|z_1|)$ and $\frac{p\gamma}{\theta}<p$, to see that
  $$(1-|z_2|)^p(1-|z_1|)^{-\frac{p\gamma}{\theta}} \leq 1$$
   and therefore
  \begin{equation} \label{eq:I_2-Kolmo}
    I_2 \leq CA |z_1 - z_2|^{p \gamma}.
  \end{equation}

  Now we estimate $I_3$. By \eqref{eq:Kolmo-assum-diff},
  \begin{align*}
    &\E |X(x_0 + z_2(1-|z_1|)^{-\frac{1}{\theta}})- X(x_0 + z_2(1-|z_2|)^{-\frac{1}{\theta}}) |^p\\
    &\leq CA|z_2|^{p \gamma}  \left|(1-|z_1|)^{-\frac{1}{\theta}} - (1-|z_2|)^{-\frac{1}{\theta}} \right|^{p \gamma}
  \end{align*}
  Because we are assuming that $(1-|z_2|)\leq (1-|z_1|)$, we can see that for $\delta \in [0,1]$ to be chosen below,
  \[\left|(1-|z_1|)^{-\frac{1}{\theta}} - (1-|z_2|)^{-\frac{1}{\theta}} \right| \leq C\left(1-|z_2|\right)^{-\frac{(1 - \delta)}{\theta}} \left|(1-|z_1|)^{-\frac{1}{\theta}} - (1-|z_2|)^{-\frac{1}{\theta}}\right|^\delta.\]
  Then because $\frac{d}{dx} (1-x)^{-\frac{1}{\theta}} = C(1-x)^{-\frac{1}{\theta} - 1}$,
  \[\left|(1-|z_1|)^{-\frac{1}{\theta}} - (1-|z_2|)^{-\frac{1}{\theta}} \right| \leq C(1-|z_2|)^{-\frac{(1 - \delta)}{\theta}}(1-|z_2|)^{-\delta\left(1 + \frac{1}{\theta} \right) }||z_1| - |z_2||^{\delta}.\]

  Using these estimates,
  \begin{equation*}
    I_3 \leq C A (1- |z_2|)^{p - \frac{\gamma p}{\theta} - \delta \gamma p} |z_1 -z_2|^{\delta \gamma p}.
  \end{equation*}
  We choose $\delta = \min\left\{1, \frac{1}{\gamma} - \frac{1}{\theta} \right\}$. This choice guarantees that $\delta \in [0,1]$ and that $p - \frac{\gamma p}{\theta} - \delta \gamma p \geq 0$. Consequently, we can conclude that
  \begin{equation} \label{eq:I_3-Kolmo}
    I_3 \leq C A |z_1 - z_2|^{\delta \gamma p}.
  \end{equation}
  Also note that by \eqref{eq:I_1-Kolmo} and \eqref{eq:I_2-Kolmo},
  \[I_1 \leq CA |z_1 - z_2|^{\delta \gamma p} \text{ and } I_2 \leq CA |z_1 - z_2|^{\delta \gamma p}.\]
  Therefore, we can conclude that for $z_1, z_2 \in B(0,1)$,
  \begin{equation} \label{eq:Y-modulus-of-cont}
    \E|Y(z_1) - Y(z_2)|^p \leq CA |z_1 - z_2|^{\gamma p \delta}.
  \end{equation}
  Similar but simpler calculations show that for any $z \in B(0,1)$,
  \begin{align}
    &\E|Y(z)|^p \leq C \left( \E|Y(z) - Y(0)|^p + \E|Y(0)|^p \right)\nonumber\\
     &\leq CA(1-|z|)^p \left( |z|^{p\gamma} (1-|z|)^{-\frac{p\gamma}{\theta}} + 1\right).
  \end{align}
  Because $p - \frac{p\gamma}{\theta}>0,$ there exists $C>0$ such that for any $z \in B(0,1)$,
  \begin{equation} \label{eq:Y-bound}
    \E|Y(z)|^p \leq CA
  \end{equation}
  and
  \begin{equation}
    \lim_{|z| \to 1} \E|Y(z)|^p  = 0,
  \end{equation}
  justifying the definition of $Y(z) = 0$ for $|z|=1$.

  By assumption, $\gamma p >d$ and  $\theta > \frac{p \gamma}{p-d}$. Therefore the exponent
  \[\gamma p \delta = \min \left\{\gamma p, p - \frac{\gamma p}{\theta} \right\}>d. \]

  Then by the Sobolev embedding theorem (Proposition \ref{prop:Sobolev-embedding}), \eqref{eq:Y-modulus-of-cont}, and \eqref{eq:Y-bound}, for any $s \in (d/p, \gamma \delta )$

  \begin{align} \label{eq:Y-Sobolev-embedding}
    &\E |Y|_{C\left(\overline{B(0,1)}\right)}^p \leq C\E | Y |_{W^{s, p }\left(\overline{B(0,1)}\right)}^p\nonumber\\
    &\leq C  \int_{B(0,1)} \E|Y(z)|^p dz + C\int_{B(0,1)}\int_{B(0,1)} \frac{\E|Y(z_1) - Y(z_2)|^p}{|z_1 - z_2|^{d + s p}}dz_1dz_2\nonumber\\
    &\leq           CA \int_{B(0,1)}dz           + CA \int_{B(0,1)}\int_{B(0,1)} |z_1-z_2|^{\gamma p \delta - sp -d}dz_1dz_2\nonumber\\
     &\leq CA.
  \end{align}

  Finally, we observe that if $x = x_0 + z(1-|z|)^{-\frac{1}{\theta}}$ then
  \[\frac{1}{1+|x-x_0|^\theta} = \frac{1}{1+ |z|^\theta(1-|z|)^{-1}} = \frac{1-|z|}{1 - |z| + |z|^\theta}.\]
  For all $\theta>0$, $\inf_{|z|\leq 1} (1 - |z| + |z|^\theta) =: c_\theta>0$. When $\theta \in (0,1)$, $c_\theta = 1$. From this and the fact that $z \mapsto x_0+ z(1-|z|)^{-\frac{1}{\theta}}$ maps $B(0,1)$ onto $\mathbb{R}^d$, we can conclude from \eqref{eq:Y-Sobolev-embedding}
  \begin{equation} \label{eq:X-weighted-Kolmo-norm}
    \E \left(\sup_{x \in \mathbb{R}^d}  \frac{|X(x)|}{1 + |x-x_0|^\theta} \right)^p \leq \frac{1}{c_\theta^p} \E \left(\sup_{z \in B(0,1)}|Y(z)| \right)^p \leq CA.
  \end{equation}

  To show that $\lim_{|x| \to 0} \frac{|X(x)|}{1 + |x-x_0|^\theta} = 0$ with probability one, choose a $\tilde{\theta} \in \left(\frac{p\gamma}{p-d},\theta\right)$. By \eqref{eq:X-weighted-Kolmo-norm},
  \[\E \left(\sup_{x \in \mathbb{R}^d} \frac{|X(x)|}{1 + |x-x_0|^{\tilde{\theta}}}\right)^p<+\infty.\]

  Therefore, with probability one,
  \[\lim_{|x| \to \infty} \frac{|X(x)|}{1 + |x-x_0|^\theta} \leq \left(\sup_{x \in \mathbb{R}^d}\frac{|X(x)|}{1 + |x-x_0|^{\tilde \theta}} \right)\lim_{|x| \to \infty} \frac{1 + |x-x_0|^{\tilde \theta}}{1 + |x-x_0|^\theta} = 0. \]
\end{proof}

The following corollary holds for processes with a time component.

\begin{corollary} \label{cor:bded-time-and-space}
  Let $X:[0,+\infty)\times \mathbb{R}^d \times \Omega \to \mathbb{R}$. Assume that for any $T>0$ there exists $A_T>0$, $p>d+1$, $\gamma\in (0,1)$ and $\beta\in(0,1)$ such that for all $x,y \in \mathbb{R}^d$ and all $t,s \in [0,T]$,
  \begin{align*}
    &\E|X(0,x_0)|^p \leq A_T\\
    &\E|X(t,x) - X(t,y)|^p \leq A_T|x-y|^{\gamma p}\\
    &\E|X(t,x) - X(s,x)|^p \leq A_T |t-s|^{\beta p}.
  \end{align*}
  Then if $p\gamma> d+ 1$, $p \beta> d+1 $ and $\theta> \frac{p\gamma}{p - (d+1)}$,
  there exists $C_{p,d,\gamma,\beta,\theta}$ such that
  \[\E \left| \sup_{t \in [0,T]} \sup_{x \in \mathbb{R}^d} \frac{|X(t,x)|}{1+|x-x_0|^\theta} \right|^p \leq C_{p,d,\gamma,\beta,\theta} A_T.\]
  and
  \[\Pro \left(\lim_{|x| \to \infty} \sup_{t \in [0,T]}  \frac{|X(t,x)|}{1 + |x-x_0|^\theta} = 0 \right) = 1.\]

  Therefore, $X \in  C_{x_0,\theta}([0,T]\times \mathbb{R}^d) $ with probability one.
\end{corollary}

\begin{proof}
  Define $Y: [0,T] \times \overline{B(0,1)}$ by
  \begin{equation}
    Y(t,z) = \begin{cases}
      (1-|z|)X(t,x_0 +z(1-|z|)^{-\frac{1}{\theta}}) & \text{ if } |z|<1\\
      0 & \text{ if } |z|=1.
    \end{cases}
  \end{equation}

  By \eqref{eq:Y-modulus-of-cont},
  \[\E|Y(t,z_1) - Y(t,z_2)|^p \leq CA_T|z_1-z_2|^{\gamma p \delta}\]
  where $\delta = \min\left\{1, \frac{1}{\gamma} - \frac{1}{\theta} \right\}$.

  We also know that
  \[\E|Y(t_1,z_1) - Y(t_2,z_1)|^p \leq A_T |t_1-t_2|^{p\beta.}\]
  Therefore,
  \[\E|Y(t_1,z_1) - Y(t_2,z_2)|^p \leq C A_T | t_1 - t_2|^{p \beta} + C A_T |z_1 - z_2|^{p \gamma \delta}.\]

  Because we assumed $p (\beta \wedge \gamma \delta) > (d+1)$, the fractional Sobolev embedding arguments prove that
  \[\E \left|\sup_{t \in [0,T]} \sup_{z \in B(0,1)} |Y(t,z)| \right|^p \leq C A_T.\]
  By \eqref{eq:X-weighted-Kolmo-norm},
  \[\E \left| \sup_{t \in [0,T]} \sup_{x \in \mathbb{R}^d} \frac{|X(t,x)|}{1 + |x-x_0|^\theta} \right|^p \leq CA_T.\]

  Finally, by choosing $\tilde \theta \in \left(\frac{p\gamma}{p - (d+1)},  \theta \right)$ we can argue that with probability one,
  \begin{align*}
    &\lim_{|x| \to \infty} \sup_{t \in [0,T]}  \frac{|X(t,x)|}{1 + |x-x_0|^\theta} \\
    &\leq \left(\sup_{t \in [0,T]} \sup_{x \in \mathbb{R}^d} \frac{|X(t,x)|}{1 + |x-x_0|^{\tilde\theta}} \right) \lim_{|x| \to \infty} \frac{1 + |x-x_0|^{\tilde\theta}}{1 + |x-x_0|^\theta} = 0.
  \end{align*}
\end{proof}

%
%
%
\subsection{Application of the Kolmogorov Continuity Theorem to stochastic convolutions.}

\begin{theorem} \label{thm:stoch-conv-growth-rate}
  Let $T>0$. Let $p> \frac{2(d+1)}{\eta}$ where $d$ is the spatial dimension and $\eta\in (0,1)$ is the parameter in the strong Dalang condition \eqref{eq:Dalang-cond}.

  For any $\theta>\frac{d+1}{p-(d+1)},$ there exists $C_{T,p,\theta}>0$ such that for all adapted
  random fields $\sigma:[0,T]\times\mathbb{R}^d \times \Omega \to \mathbb{R}$ satisfying
\begin{equation}
  \sup_{t \in [0,T]}\sup_{x \in \mathbb{R}^d} \E |\sigma(t,x)|^p < +\infty,
\end{equation}
  the stochastic convolution
  \[Z(t,x) := \int_0^t \int_{\mathbb{R}^d} G(t-s,x-y) \sigma(s,y)W(ds,dy)\]
  satisfies

  \begin{align} \label{eq:stoch-conv-bound}
  &\sup_{x_0 \in \mathbb{R}^d}\E \left|\sup_{t \in [0,T]}\sup_{x \in \mathbb{R}^d} \frac{|Z(t,x)|}{1 + |x-x_0|^\theta}  \right|^p 
   \leq C_{T,p,\theta} \sup_{t \in [0,T]} \sup_{x_0 \in \mathbb{R}^d} \E \left|  \sigma(t,x_0)\right|^p
  \end{align}
  and
  \begin{equation}
    \Pro(Z \in C_{x_0,\theta}([0,T]\times \mathbb{R}^d) )= 1 \text{ for all } x_0 \in \mathbb{R}.
  \end{equation}
  Furthermore, $C_{T,p,\theta}$ has the property that for any fixed $p, \theta$,
  \begin{equation} \label{eq:C-small}
    \lim_{T \to 0} C_{T,p,\theta} = 0.
  \end{equation}
\end{theorem}

\begin{proof}
  By Theorem 2.1 of \cite{sss-1999}, for any $\alpha \in \left(0, \frac{\eta}{2} \right)$, $\gamma \in (0, 2\alpha)$, $\beta \in (0,\alpha)$, and $p \geq 2$, there exists constants $C_{\alpha,p,\gamma,\beta}>0$ such that for any $x, x_1, x_2 \in \mathbb{R}^d$ and $t,t_1,t_2 \in [0,T]$,
  \begin{align}
    &\E|Z(t,x)|^p \leq C_{\alpha,p,\gamma,\beta}  T^{p\alpha}\sup_{s \in [0,T]} \sup_{y \in \mathbb{R}^d} \E |\sigma(s,y)|^p , \\
    &\E|Z(t,x_1) - Z(t,x_2)|^p \nonumber\\
    &\qquad\leq C_{\alpha,p,\gamma,\beta} |x_1 - x_2|^{\gamma p}  T^{p\left(\alpha - \frac{\gamma}{2}\right)} \sup_{s \in [0,T]} \sup_{y \in \mathbb{R}^d}\E |\sigma(s,y)|^p ,\\
    &\E|Z(t_1,x) - Z(t_2,x)|^p \nonumber \\
    &\qquad\leq C_{\alpha,p,\gamma,\beta} |t_1-t_2|^{\beta p} T^{p\left( \alpha -\beta\right)}
     \sup_{s \in [0,T]}\sup_{y \in \mathbb{R}^d}\E |\sigma(s,y)|^p.
  \end{align}

  Let $\theta>\frac{d+1}{p-(d+1)}$ be given. Because $p> \frac{2(d+1)}{\eta}$, we choose $\alpha \in \left(\frac{d+1}{p}, \frac{\eta}{2} \right)$, $\beta \in \left(\frac{d+1}{p}, \alpha \right)$, and $\gamma \in \left(\frac{d+1}{p},2\alpha \wedge \frac{\theta(p-(d+1))}{p}  \right)$.  These choices guarantee that $\gamma p >d+1$, $\beta p >d+1$, and $\theta> \frac{\gamma p}{p-(d+1)}$.
  By Corollary \ref{cor:bded-time-and-space}, for  $T>0$, and $x_0 \in \mathbb{R}^d$

  \begin{align*}
    &\E \left[\sup_{t \in [0,T]} \sup_{x \in \mathbb{R}^d} \frac{|Z(t,x)|}{1+|x-x_0|^\theta} \right]^p \nonumber\\
    &\leq C_{p,\alpha,\gamma,\beta,\theta} \left(T^{p\alpha} + T^{p\left(\alpha - \gamma/2 \right)} + T^{p(\alpha - \beta)} \right) \sup_{s \in [0,T]} \sup_{y \in \mathbb{R}^d} \E |\sigma(s,y)|^p .
  \end{align*}
  The right-hand side is independent of $x_0$. Therefore
  \begin{align} \label{eq:Z-bound-w-T}
    &\sup_{x_0 \in \mathbb{R}^d}\E \left[\sup_{t \in [0,T]} \sup_{x \in \mathbb{R}^d} \frac{|Z(t,x)|}{1+|x-x_0|^\theta} \right]^p \nonumber\\
    &\leq C_{p,\alpha,\gamma,\beta,\theta} \left(T^{p\alpha} + T^{p\left(\alpha - \gamma/2 \right)} + T^{p(\alpha - \beta)} \right) \sup_{s \in [0,T]} \sup_{y \in \mathbb{R}^d} \E |\sigma(s,y)|^p .
  \end{align}

  Finally we can verify \eqref{eq:C-small} from \eqref{eq:Z-bound-w-T}.
%

\end{proof}

\begin{corollary}
  If $\sigma: [0,T] \times \mathbb{R}^d \times \Omega \to \mathbb{R}$ is an adapted random field such that
  \begin{equation}
    \sup_{t \in [0,T]} \sup_{x \in \mathbb{R}^d} \E|\sigma(t,x)|^p < + \infty
  \end{equation}
  for all $p \geq 1$,
  then
  \begin{equation}
    Z(t,x) = \int_0^t \int_{\mathbb{R}^d} G(t-s,x-y)\sigma(s,y)W(dyds)
  \end{equation}
  belongs to $C_{x_0,\theta}([0,T]\times \mathbb{R}^d)$ with probability one for any $x_0 \in \mathbb{R}^d$ and any $\theta>0$.

  In particular, this corollary holds when $\sigma$ is uniformly bounded.
\end{corollary}

\begin{proof}
  Theorem \ref{thm:stoch-conv-growth-rate} holds for any $\theta> \frac{d+1}{p-(d+1)}$. If $p$ can be arbitrarily large, then $\theta$ can be arbitrarily small.
\end{proof}

\section{The mapping $\mathcal{M}$} \label{S:M}
Given a deterministic
$z \in C_b([0,T]\times \mathbb{R}^d)$, let $u= \mathcal{M}(z)$ be the solution to the integral equation

\begin{equation} \label{eq:M-def}
  u(t,x) = \int_0^t \int_{\mathbb{R}^d} G(t-s,x-y) f(u(s,y))dyds + z(t,x).
\end{equation}

In this section we will show that for any $z \in C_b([0,T]\times \mathbb{R}^d)$, the solution to \eqref{eq:M-def} exists and is unique and that $\mathcal{M}$ can be uniquely extended to a map from  $C_{x_0,\theta}([0,T]\times \mathbb{R}^d)$ to itself for any $x_0 \in \mathbb{R}^d$ and $\theta>0$. We will also prove that $\mathcal{M}$ is a Lipschitz continuous mapping from $C_{x_0,\theta}([0,T]\times \mathbb{R}^d)$ to itself for any $T>0$, $x_0 \in \mathbb{R}^d$, and $\theta>0$, and that the Lipschitz constant does not depend on the center $x_0$ (Theorem \ref{thm:M-extension}).

\begin{lemma} \label{lem:a-priori1}
  Whenever $z \in C_b([0,T]\times \mathbb{R}^d)$ and $u$ is a solution to \eqref{eq:M-def}, it follows that $u \in C_b([0,T]\times \mathbb{R}^d)$ and
  \begin{equation} \label{eq:a-priori}
    \sup_{t \in [0,T]} \sup_{x \in \mathbb{R}^d} |u(t,x)| \leq \left(\frac{ e^{(1 + 2\kappa)T}-1}{1+2\kappa}\right)\sup_{t \in [0,T]} \sup_{x \in \mathbb{R}^d} |f(z(t,x))| + \sup_{t \in [0,T]} \sup_{x \in \mathbb{R}^d}|z(t,x)|
  \end{equation}
  where $\kappa$ is from \eqref{eq:f-cond-dissip} and $\frac{ e^{(1 + 2\kappa)T}-1}{1+2\kappa} :=T$ when $\kappa=-\frac{1}{2}$.
\end{lemma}

\begin{proof}
    Assume that $u$ is a solution to \eqref{eq:M-def} for $z \in C_b([0,T]\times \mathbb{R}^d)$. Let $v(t,x) := u(t,x) - z(t,x)$. Note that
  \begin{align*}
    &v(t,x) = \int_0^t \int_{\mathbb{R}^d} G(t-s,x-y)f(u(s,y))dyds. \\
  \end{align*}
  $v(t,x)$ is differentiable and solves
  \begin{equation}
    \frac{\partial v}{\partial t}(t,x) = \frac{1}{2}\Delta v(t,x) + f(v(t,x) + z(t,x)), \ \ \ v(0,x) = 0.
  \end{equation}
  We can re-write this as
  \[\frac{\partial v}{\partial t}(t,x) =\frac{1}{2}\Delta v(t,x) + (f(v(t,x) + z(t,x)) - f(z(t,x))) + f(z(t,x)).\]
  We can calculate that
  \begin{align} \label{eq:deriv-v-square}
    \frac{\partial }{\partial t} (v(t,x))^2 = & v(t,x) \Delta v(t,x) \nonumber\\
    &+ 2(f(v(t,x) + z(t,x)) - f(z(x,t)))v(t,x) \nonumber\\
    &+ 2f(z(t,x))v(t,x)
  \end{align}
  Also note that
  \begin{equation} \label{eq:Laplace-of-squares}
   \frac{1}{2}\Delta \left[ (v(t,x))^2\right] =  v(t,x) \Delta v(t,x) +  |\nabla v(t,x)|^2 \geq  v(t,x) \Delta v(t,x).
  \end{equation}
  Furthermore, \eqref{eq:f-cond-sign} guarantees that
  \begin{align} \label{eq:f-diff-squared}
    &(f(v(t,x) + z(t,x)) - f(z(t,x)))v(t,x) \leq \kappa |v(t,x)|^2.
  \end{align}
  By Young's inequality,
  \begin{equation} \label{eq:Youngs}
    f(z(t,x))v(t,x) \leq \frac{1}{2}|f(z(t,x))|^2 + \frac{1}{2} |v(t,x)|^2.
  \end{equation}
  Estimates  \eqref{eq:deriv-v-square}, \eqref{eq:Laplace-of-squares}, \eqref{eq:f-diff-squared}, and \eqref{eq:Youngs} show that $(v(t,x))^2$ satisfies
  \[\frac{\partial}{\partial t} (v(t,x))^2 \leq \frac{1}{2} \Delta(v(t,x))^2 + (1 + 2\kappa)|v(t,x)|^2 + |f(z(t,x))|^2.\]
  Then it holds that
  \[\frac{\partial}{\partial t} \left(e^{-(1 + 2\kappa)t}(v(t,x))^2 \right) \leq \frac{1}{2}\Delta \left(e^{-(1 + 2\kappa)t}(v(t,x))^2 \right) + e^{-(1 + 2\kappa)t}  |f(z(t,x))|^{2}.\]
  This is a subsolution to a heat equation. By the comparison principle of the heat equation
  \begin{align*}
    &e^{-(1 + 2\kappa)t}(v(t,x))^2 \leq \int_0^t \int_{\mathbb{R}^d} G(t-s,x-y) e^{-(1 + 2\kappa)s} |f(z(s,y))|^{2}dyds.
  \end{align*}
  Because we assumed that $\sup_{s \in [0,T]}\sup_{y \in \mathbb{R}^d} |z(s,y)|< + \infty$ and $G$ is the heat kernel, we can conclude that
  \[\sup_{t \in [0,T]} \sup_{x \in \mathbb{R}^d} |v(t,x)|^2 < \left(e^{(1 + 2 \kappa)T} \int_0^T e^{-(1 + 2\kappa)s}ds \right) \sup_{t \in [0,T]}\sup_{x \in \mathbb{R}^d} |f(z(t,x))|^2.\]

  Finally, \eqref{eq:a-priori} follows because $u(t,x) = v(t,x) + z(t,x)$. The continuity of $u$ is a straightforward consequence of the fact that $$(t,x) \mapsto \int_0^t \int_{\mathbb{R}^d} G(t-s,x-y)\varphi(s,y)dyds$$ is continuous whenever $\varphi$ is uniformly bounded (see, for example, \cite{sss-1999}).
\end{proof}

\begin{theorem}[Existence of $\mathcal{M}(z)$]
  Given $z \in C_b([0,T]\times \mathbb{R}^d)$ there exists a $C_b([0,T]\times \mathbb{R}^d)$-valued solution $u =\mathcal{M}(z)$ to \eqref{eq:M-def}.
\end{theorem}

\begin{proof}
  By Proposition \ref{prop-f-cond-equiv},
  \[\phi(u) := f(u) - \kappa u\]
  is a non-increasing function.

  Define the  Yosida approximations for $\phi$ by
  \begin{equation}
    \phi_n(u) = n\left(\left(I -\frac{1}{n}\phi\right)^{-1} (u) - u\right), \ \ \ \text{where } I(u) = u .
  \end{equation}
  According to \cite[Proposition D.11]{dpz},
  \begin{enumerate}
    \item Each $\phi_n$ is Lipschitz continuous satisfying
  \[|\phi_n(u_2) - \phi_n(u_1)| \leq 2n | u_2 - u_1|, \text{ for all } u_1,u_2 \in \mathbb{R}.\]
    \item Each $\phi_n$ is non-increasing.
    \item For every $u \in \mathbb{R}$, $\displaystyle{\lim_{n \to \infty} \phi_n(u) = \phi(u),}$ and
    \item Each $\phi_n$ is pointwise bounded by $\phi$ in the sense that for all $u \in \mathbb{R}$, $\displaystyle{|\phi_n(u)| \leq |\phi(u)|.}$
  \end{enumerate}
  Let $f_n(u):= \phi_n(u) + \kappa u$. The $f_n$ satisfy
  \begin{align}
    & |f_n(u_1) - f_n(u_2)| \leq (2n+ |\kappa|)|u_1 - u_2|,  &&\text{ for all } u_1, u_2 \in \mathbb{R} \label{eq:f_n-lip}\\
    &f_n(u_2) - f_n(u_1) \leq \kappa(u_2 - u_1) &&\text{ for all } u_1<u_2, \label{eq:f_n-dissip}\\
    &\lim_{n \to \infty} f_n(u) = f(u), &&\text{ for all } u \in \mathbb{R}, \label{eq:f_n-lim}\\
    & |f_n(u)| \leq |\phi(u)| + |\kappa| |u| \leq |f(u)| + 2 |\kappa| |u|, &&\text{ for all } u \in \mathbb{R}. \label{eq:f_n-bound}
  \end{align}
  Because the $f_n$ are Lipschitz continuous \eqref{eq:f_n-lip}, standard Picard iteration arguments  prove that there exists a unique solution for each $n$ to
  \[u_n(t,x) = \int_0^t \int_{\mathbb{R}^d} G(t-s,x-y) f_n(u_n(s,y))dyds + z(t,x).\]

  By Lemma \ref{lem:a-priori1}, because all of the $f_n$ satisfy \eqref{eq:f_n-dissip} with the same constant $\kappa$,
  \begin{align*}
    &\sup_n\sup_{t \in [0,T]} \sup_{x \in \mathbb{R}} |u_n(t,x)|\\
    &\leq \left(\frac{e^{(1 + 2 \kappa)T} -1 }{1 + 2\kappa} \right) \sup_{t \in [0,T]}\sup_{x \in \mathbb{R}^d} |f_n(z(t,x))| + \sup_{t \in [0,T]} \sup_{x \in \mathbb{R}^d}|z(t,x)|.
  \end{align*}
  By \eqref{eq:f_n-bound},
  \begin{align*}
    &\sup_n\sup_{t \in [0,T]} \sup_{x \in \mathbb{R}} |u_n(t,x)|\\
    &\leq \left(\frac{e^{(1 + 2 \kappa)T} -1 }{1 + 2\kappa} \right) \sup_{t \in [0,T]}\sup_{x \in \mathbb{R}^d} |f(z(t,x))| \\
    &\qquad+ \left(\frac{2|\kappa|\left( e^{(1+2\kappa)T} -1 \right)}{1 + 2\kappa} + 1\right) \sup_{t \in [0,T]} \sup_{x \in \mathbb{R}^d} |z(t,x)|\\
    &<+\infty.
  \end{align*}
  As a consequence of the fact that $u_n$ are uniformly bounded and \eqref{eq:f_n-bound},
  \[\sup_n\sup_{t \in[0,T]} \sup_{x \in \mathbb{R}^d} |f_n(u_n(t,x))|<+\infty \]
  as well.
  Let $v_n(t,x) = u_n(t,x) -z(t,x)$. $v_n(t,x)$ solves
  \[v_n(t,x) = \int_0^t \int_{\mathbb{R}^d} G(t-s,x-y) f_n(u_n(s,y))dyds,\]
  and now using well-known properties of convolutions of bounded functions against the heat kernel (see, for example \cite{sss-1999}), there exist $C>0$ and $\gamma \in (0,1)$ such that
  \begin{align}
    &\sup_{n \in \mathbb{N}} \sup_{x\in \mathbb{R}^d} \sup_{t \in [0,T]} |v_n(t,x)| \leq C\\
    &\sup_{n \in \mathbb{N}} \sup_{x_1\in \mathbb{R}^d}\sup_{x_2 \in  \mathbb{R}^d} \sup_{t_1 \in [0,T]} \sup_{t_2 \in [0,T]} \frac{|v_n(t_1,x_1) - v_n(t_2,x_2)|}{|t_1 - t_2|^\gamma + |x_1 - x_2|^\gamma} \leq C.
  \end{align}
  Therefore, by the Arzela-Ascoli theorem, there is a subsequence (relabeled $v_n$) that converges uniformly on compact subsets of $[0,T]\times \mathbb{R}^d$. We call the limit $\tilde{v}$. By the dominated convergence theorem and the fact that \eqref{eq:f_n-lim} guarantees that $f_n(v_n(s,y) + z(s,y)) \to f(\tilde{v}(s,y) + z(s,y))$,
  \[\tilde{v}(t,x) = \int_0^t \int_{\mathbb{R}^d} G(t-s,x-y)f(\tilde{v}(s,y) + z(s,y))dsdy. \]
  Let $\tilde{u}(t,x):= \tilde{v}(t,x) + z(t,x)$ and notice that $\tilde u$ solves \eqref{eq:M-def}.
\end{proof}

\begin{lemma}[A Gr\"onwall-type lemma] \label{lem:Gronwall-fast}
  Let $C_1,C_2>0$ and let $\psi(t)$ and $\Theta(t)$ be positive functions. Assume that $\varphi(t)$ is a positive function with the properties that
  \begin{equation}
    \varphi(0) = 0
  \end{equation}
  and for all $t>0$,
  \begin{equation} \label{eq:gronwall-fast-assum}
    \frac{d^-}{dt}\varphi(t) \leq   C_1\varphi(t) + C_2   + \Theta(t) \mathbbm{1}_{\{\varphi(t)< \psi(t)\}}.
  \end{equation}
  Then for any $T>0$,
  \begin{equation} \label{eq:gronwall-fast}
    \sup_{t \in [0,T]} \varphi(t) \leq \left( C_2 T + \sup_{t \in [0,T]} \psi(t) \right) e^{C_1T}  .
  \end{equation}
\end{lemma}
Notice that in this lemma there are no bounds on the magnitude of $\Theta$ and that $\Theta$ does not affect the bound of the supremum of $\varphi$ in \eqref{eq:gronwall-fast}. The assumption \eqref{eq:gronwall-fast-assum} states that $\varphi$ can grow very quickly if $\varphi(t)<\psi(t)$, but $\varphi(t)$ can only grow in at exponential speed once $\varphi(t)$ exceeds $\psi(t)$.
\begin{proof}
  For any $t \in [0,T]$, define
  \[\tau_t: = \sup\left\{s \in [0,t]: \varphi(s) \leq \sup_{r \in [0,T]} \psi(r)\right\}.\]
  If $\varphi(t) \leq \sup_{r \in [0,T]} \psi(r)$, then $\tau_t = t$.
  We observe that $\varphi(\tau_t) \leq \sup_{r \in [0,T]} \psi(r)$ while
  for $s \in (\tau_t,t]$, $\varphi(t) > \sup_{r \in [0,T]} \psi(r)$. Then by assumption \eqref{eq:gronwall-fast-assum}, for $s \in (\tau_t,t]$,
  \[\varphi(s) \leq \varphi(\tau_t) + \int_{\tau_t}^t (C_2 + C_1 \varphi(r))dr\leq \sup_{r \in [0,T]} \psi(r) + \int_{\tau_t}^t (C_2 + C_1 \varphi(r))dr.\]
  By Gr\"onwall's lemma and the fact that $t-\tau_t \leq T$,
  \[\varphi(t) \leq \left(C_2(t-\tau_t) +\sup_{r \in [0,T]} \psi(r)  \right)e^{C_1(t-\tau_t)} \leq \left( C_2 T + \sup_{r \in [0,T]}\psi(r) \right)e^{C_1 T}.\]
  Because our choice of $t \in [0,T]$ was arbitrary, we can conclude that \eqref{eq:gronwall-fast} holds.
\end{proof}

\begin{theorem}
  For any $\theta>0$ there exists a constant $C=C(\theta,\kappa)>0$ such that whenever $T>0$ and $z_1,z_2 \in C_b([0,T]\times \mathbb{R}^d)$ and $u_1 = \mathcal{M}(z_1)$ and $u_2 = \mathcal{M}(z_2)$ are solutions to \eqref{eq:M-def} then for any $x_0 \in \mathbb{R}^d$
  \begin{align}
    &\sup_{t \in [0,T]} \sup_{x \in \mathbb{R}^d} 
    \frac{|u_1(t,x) - u_2(t,x)|}{1 + |x-x_0|^\theta} 
    \leq Ce^{CT} \sup_{t \in [0,T]} \sup_{x \in \mathbb{R}^d} 
    \frac{|z_1(t,x) - z_2(t,x)|}{1 + |x-x_0|^\theta}.  \label{eq:M-Lip-w-proof}
  \end{align}
  The constant $C$ is independent of $x_0$ and $T$ and depends only on the constant $\kappa$ in \eqref{eq:f-cond-dissip} and $\theta>0$.

  In particular, when $z_1 = z_2$ this theorem proves the uniqueness of solutions to \eqref{eq:M-def}.
\end{theorem}

\begin{proof}
  Let $z_1, z_2 \in C_b([0,T]\times \mathbb{R}^d)$. Let $u_1 =\mathcal{M}(z_1)$ and $u_2 = \mathcal{M}(z_2)$ be the solutions to \eqref{eq:M-def}. By Lemma \ref{lem:a-priori1}, $u_1\in C_b([0,T]\times \mathbb{R}^d)$ and $u_2  \in C_b([0,T]\times \mathbb{R}^d)$. Let $v_1 = u_1 - z_1$ and $v_2 = u_2-z_2$. $v_1$ and $ v_2$ are differentiable and
  \begin{equation}
    \frac{\partial v_i}{\partial t }(t,x) = \frac{1}{2}\Delta v_i(t,x) + f(v_i(t,x) + z_i(t,x)),  \ \ \ i=1,2.
  \end{equation}

  We will use the weights $(1 + |x-x_0|^2)^{-\frac{\theta}{2}}$ because they are differentiable in $x$ and there exists a constant $C>0$ such that
  \begin{equation} \label{eq:weights-equiv}
    \frac{1}{C(1 + |x-x_0|^\theta)} \leq (1 + |x-x_0|^2)^{-\frac{\theta}{2}} \leq \frac{C}{1 + |x-x_0|^\theta}, \ \ \ x \in \mathbb{R}^d.
  \end{equation}
  Let
  \begin{align}
    \label{eq:z-tilde} \tilde{z}(t,x):= (1 + |x-x_0|^2)^{-\frac{\theta}{2}} (z_1(t,x) - z_2(t,x))\\
    \label{eq:v-tilde} \tilde{v}(t,x):= (1 + |x-x_0|^2)^{-\frac{\theta}{2}}(v_1(t,x) - v_2(t,x)).
  \end{align}
  This weighted difference $\tilde{v}$ solves the PDE
  \begin{align} \label{eq:tilde-v-PDE}
    \frac{\partial \tilde{v}}{\partial t}(t,x) = &\frac{1}{2}\Delta \tilde{v}(t,x) + \theta \frac{(x-x_0)\cdot \nabla \tilde{v}(t,x)}{1 + |x-x_0|^2} \nonumber\\
    &+ \frac{\theta \left(d + (d + \theta -2)|x-x_0|^2\right)}{2(1 + |x-x_0|^2)^2} \tilde{v}(t,x) \nonumber\\
    &+ (1 + |x-x_0|^2)^{-\frac{\theta}{2}}(f(v_1(t,x) + z_1(t,x)) - f(v_2(t,x) + z_2(t,x))).
  \end{align}

  We proved in Lemma \ref{lem:a-priori1} that $v_1,v_2 \in C_b([0,T]\times \mathbb{R}^d)$. Therefore, $$\lim_{|x| \to \infty} |\tilde{v}(t,x)| = 0$$ and $\tilde v(t,\cdot)$ belong to $C_0(\mathbb{R}^d)$, the space of functions that vanish at infinity.

  By Proposition \ref{prop-left-deriv-subdiff}, the upper-left-derivative of the norm $|\tilde{v}(t,\cdot)|_{C_0}$ is bounded by
  \begin{equation} \label{eq:norm-deriv-bound}
    \frac{d^-}{d t} |\tilde v(t,\cdot)|_{C_0} \leq \left< \frac{\partial \tilde{v}}{\partial t}(t,\cdot), \mu_t \right>_{C_0,C_0^\star}
  \end{equation}
  where $\mu_t \in \partial |v(t,\cdot)|_{C_0}$, the subdifferential of $\tilde{v}(t,\cdot)$ given by \eqref{eq:subdiff}. By Proposition \ref{prop:subdff}, $\mu_t(dx) = \sgn(\tilde v(t,x))\tilde{\mu}_t(dx)$ where $\tilde{\mu}_t$ is a positive unit measure supported on the set $\{x \in \mathbb{R}^d: |\tilde v(t,x)| = | \tilde v(t,\cdot)|_{C_0}\}$. Because \eqref{eq:norm-deriv-bound} holds for any $\mu_t \in \partial |\tilde v(t,\cdot)|_{C_0}$,  we can take $\mu_t = \delta_{x_t} \sgn(\tilde v(t,x_t))$ where $x_t \in \mathbb{R}^d$ is a global maximizer or minimizer ($|\tilde v(t,\cdot)|_{C_0} = |\tilde v(t,x_t)| $).

  By \eqref{eq:norm-deriv-bound} and \eqref{eq:tilde-v-PDE},
  \begin{align} \label{eq:v-norm-bound}
    &\frac{d^-}{dt} |\tilde v(t,\cdot)|_{C_0} \nonumber \\
    &\leq  \sgn(\tilde v(t,x_t))\frac{1}{2}\Delta v(t,x_t) +  \sgn(\tilde v(t,x_t)) \frac{\theta(x_t-x_0)\cdot \nabla \tilde{v}(t,x_t)}{1 + |x_t- x_0|^2} \nonumber\\
    & \hspace{1.6cm}+  \frac{\theta(d + (d+\theta-2)|x_t-x_0|^2)}{2(1 + |x_t-x_0|^2)^2}\tilde v(t,x_t)\sgn(\tilde v(t,x_t))\nonumber\\
    & \hspace{1.6cm}+ (1 +|x_t-x_0|^2)^{-\frac{\theta}{2}}\sgn(\tilde{v}(t,x_t))\nonumber\\
    & \hspace{2.4cm}\times(f(v_1(t,x_t) +z_1(t,x_t)) - f(v_2(t,x_t) + z_2(t,x_t))) .
  \end{align}
  Because $\tilde{v}$ is differentiable and $x_t$ is either a global maximizer or minimizer,
  \begin{align*}
    &\sgn(\tilde v(t,x_t))\Delta \tilde{v}(t,x_t) \leq 0,\\
    &\nabla \tilde{v}(t,x_t) = 0, \text{ and }\\
    &v(t,x_t)\sgn(\tilde v(t,x_t)) = |v(t,\cdot)|_{C_0}.
  \end{align*}
  Furthermore, we note that
  \[\sup_{x \in \mathbb{R^d}}\frac{\theta(d + (d+\theta-2)|x-x_0|^2)}{2(1 + |x-x_0|^2)^2} <+\infty. \]

  Next we analyze the term $$(1 +|x_t-x_0|^2)^{-\frac{\theta}{2}}(f(v_1(t,x_t) + z_1(t,x_t)) - f(v_2(t,x_t) + z_2(t,x_t)))\sgn(\tilde{v}(t,x)).$$
  Observe that $$\sgn(\tilde{v}(t,x_t)) = \sgn(v_1(t,x_t) - v_2(t,x_t))$$ and
  $$\sgn(\tilde{v}(t,x_t)  + \tilde z(t,x_t)) =\sgn(v_1(t,x_t) + z_1(t,x_t) -(v_2(t,x_t)+ z_2(t,x_t))).$$
  We divide this into two possibilities. First, if $\sgn(\tilde{v}(t,x_t) + \tilde{z}(t,x_t)) = \sgn(\tilde{v}(t,x_t))$, then by \eqref{eq:f-cond-sign},
  \begin{align*}
    &(1 +|x_t-x_0|^2)^{-\frac{\theta}{2}}(f(v_1(t,x_t) + z_1(t,x_t)) - f(v_2(t,x_t) + z_2(t,x_t)))\sgn(\tilde{v}(t,x_t)) \\
    &=(1 +|x_t-x_0|^2)^{-\frac{\theta}{2}}(f(v_1(t,x_t) + z_1(t,x_t)) - f(v_2(t,x_t) + z_2(t,x_t)))\\
    &\qquad\qquad\times \sgn(\tilde{v}(t,x_t) + \tilde{z}(t,x_t))\\
    &=(1 +|x_t-x_0|^2)^{-\frac{\theta}{2}}(f(v_1(t,x_t) + z_1(t,x_t)) - f(v_2(t,x_t) + z_2(t,x_t)))\\
    &\qquad\qquad\times \sgn(v_1(t,x_t) + z_1(t,x_t) -v_2(t,x_t) - z_2(t,x_t))\\
    &\leq (1 +|x_t-x_0|^2)^{-\frac{\theta}{2}}\kappa|v_1(t,x_t) + z_1(t,x_t) -(v_2(t,x_t) + z_2(t,x_t))| \\
    &\leq \kappa|\tilde{v}(t,x_t) + \tilde{z}(t,x_t)|\\
    &\leq \kappa |\tilde{v}(t,\cdot)|_{C_0} + \kappa |\tilde{z}(t,\cdot)|_{C_0}.
  \end{align*}
  If $\sgn(\tilde{v}(t,x_t)) \not = \sgn(\tilde{v}(t,x_t) + \tilde{z}(t,x_t))$ then it must hold that
  $$ |\tilde{v}(t,x_t)|< |\tilde{z}(t,x_t)| .$$
  Because $|v(t,x_t)| = |v(t,\cdot)|_{C_0}$, a consequence of this inequality is that
  $$  |\tilde{v}(t,\cdot)|_{C_0} =  |\tilde{v}(t,x_t)| < |\tilde{z}(t,x_t)| \leq |\tilde{z}(t,\cdot)|_{C_0} .$$ Therefore,
  \begin{align} \label{eq:f-diff-bound}
    &(1 +|x_t-x_0|^2)^{-\frac{\theta}{2}}(f(v_1(t,x_t) + z_1(t,x_t)) - f(v_2(t,x_t) + z_2(t,x_t)))\sgn(\tilde{v}(t,x)) \nonumber\\
    &\leq \kappa|\tilde{v}(t,\cdot)|_{C_0} + \kappa|\tilde{z}(t,\cdot)|_{C_0} + \Phi(t) \mathbbm{1}_{\{|\tilde{v}(t,\cdot)|_{C_0} < |\tilde{z}(t,\cdot)|_{C_0}\}}
  \end{align}
  where
  \[\Phi(t) := (1 +|x_t-x_0|^2)^{-\frac{\theta}{2}}|f(v_1(t,x_t) + z_1(t,x_t)) - f(v_2(t,x_t) + z_2(t,x_t))|.\]

  We plug all of these observations into \eqref{eq:v-norm-bound},
  \begin{align*}
    &\frac{d^-}{dt} |\tilde{v}(t,\cdot)|_{C_0}
    \leq C|\tilde{v}(t,\cdot)|_{C_0} + C|\tilde{z}(t,\cdot)|_{C_0} + \Phi(t) \mathbbm{1}_{\{|v(t,\cdot)|_{C_0} <|\tilde{z}(t,\cdot)|_{C_0}\}}.
  \end{align*}
  By Lemma \ref{lem:Gronwall-fast} we can conclude that there exists $C>0$ such that
  \begin{equation}
    \sup_{t \in [0,T]} |\tilde{v}(t,\cdot)|_{C_0} \leq C e^{CT} \left(  \sup_{t \in [0,T]}|\tilde{z}(t,\cdot)|_{C_0} \right).
  \end{equation}

  By the definitions of $\tilde{v}$ and $\tilde{z}$, we have shown that
  \begin{align}
    &\sup_{t \in [0,T]} \sup_{x \in \mathbb{R}^d} (1 + |x-x_0|^2)^{-\frac{\theta}{2}}|v_1(t,x) - v_2(t,x)| \nonumber\\
    &\leq Ce^{CT} \sup_{t \in [0,T]} \sup_{x \in \mathbb{R}^d} (1 + |x-x_0|^2)^{-\frac{\theta}{2}}|z_1(t,x) - z_2(t,x)|.
  \end{align}

  Finally, the result follows because $u_i= v_i + z_i$ and because of \eqref{eq:weights-equiv}.
\end{proof}

\begin{corollary} \label{cor:M-lin-growth}
  For any $\theta>0$ there exists $C=C(\kappa,\theta)>0$ such that for any $z \in C_b([0,T]\times \mathbb{R}^d)$,
  \begin{align} \label{eq:M-a-priori}
    \sup_{t \in [0,T]} \sup_{x \in \mathbb{R}^d} \frac{|\mathcal{M}(z)(t,x)|}{1 + |x-x_0|^\theta}
    \leq C e^{C T} \left( 1 +  \sup_{t \in [0,T]}\sup_{x \in \mathbb{R}^d} \frac{|z(t,x)|}{1 + |x-x_0|^\theta}\right).
  \end{align}
\end{corollary}

\begin{proof}
  Let $u_1 = \mathcal{M}(0)$, where $0 \in C_b([0,T]\times\mathbb{R}^d)$ denotes the constant zero function. By Lemma \ref{lem:a-priori1},
  \begin{equation}
    \sup_{t \in [0,T]}\sup_{x \in \mathbb{R}^d}\frac{|u_1(t,\cdot)|}{1 + |x-x_0|^\theta} \leq \left(\frac{e^{(2\kappa +1)T}-1}{2\kappa +1}\right)f(0).
  \end{equation}

  Let $u = \mathcal{M}(z)$ for some $z \in C_{b}([0,T]\times \mathbb{R}^d)$. Then by \eqref{eq:M-Lip-w-proof}, there exists $C>0$ such that
  \begin{align*}
    &\sup_{t \in [0,T]}\sup_{x \in \mathbb{R}^d}\frac{|u(t,x)|}{1 + |x-x_0|^\theta}\\
     &\leq \sup_{t \in [0,T]}\sup_{x \in \mathbb{R}^d} \frac{|u(t,x) - u_1(t,x)|}{1 + |x-x_0|^\theta} + \sup_{t \in [0,T]}\sup_{x \in \mathbb{R}^d} \frac{|u_1(t,x)|}{1 + |x-x_0|^\theta}\nonumber\\
    &\leq Ce^{CT}\sup_{t \in [0,T]} \sup_{x \in\mathbb{R}^d} \frac{|z(t,x)-0|}{1 + |x-x_0|^\theta} + Ce^{CT}.
  \end{align*}
\end{proof}
\begin{theorem} \label{thm:M-extension}
  For any $x_0 \in \mathbb{R}^d$ and $\theta>0$, $\mathcal{M}$ can be uniquely extended to a continuous function from $C_{x_0,\theta}([0,T]\times \mathbb{R}^d) \to C_{x_0,\theta}([0,T]\times \mathbb{R}^d)$.
  The extension satisfies both  \eqref{eq:M-Lip-w-proof} and \eqref{eq:M-a-priori} for all $z_1,z_2,z \in C_{x_0,\theta}([0,T]\times \mathbb{R}^d)$.
  If $\theta \in (0,2/\nu)$, where $\nu$ is the constant in \eqref{eq:f-growth}, then for all $z \in C_{x_0,\theta}([0,T]\times \mathbb{R}^d)$,  $u=\mathcal{M}(z)$ satisfies the integral equation \eqref{eq:M-def}.
  If $\theta \geq 2/\nu$, then it is possible that the integrand $$y \mapsto G(t-s,x-y)f(\mathcal{M}(z)(s,y))$$ may not be integrable.
\end{theorem}

\begin{proof}
  First we argue that $C_b([0,T]\times \mathbb{R}^d)$ is dense in $C_{x_0,\theta}([0,T]\times \mathbb{R}^d)$ for any $\theta>0$. Given $\varphi \in C_{x_0,\theta}([0,T]\times \mathbb{R}^d)$, let
  \begin{equation}
    \varphi_n(t,x) =
    \begin{cases}
      \varphi(t,x) & \text{ if } |x-x_0|\leq n\\
      \displaystyle{\frac{\varphi(t,x)}{1 + (|x-x_0| - n)^{\theta}}}& \text{ if } |x-x_0|>n.
    \end{cases}
  \end{equation}

  Each $\varphi_n$ is bounded. In fact, because $\varphi \in C_{x_0,\theta}([0,T]\times \mathbb{R}^d)$,
  \[\lim_{|x|\to \infty}\sup_{t \in [0,T]} |\varphi_n(t,x)| = 0.\]

  Furthermore, we can check that
  \begin{align}
    \sup_{t \in [0,T]} \sup_{x \in \mathbb{R}^d}\frac{|\varphi_n(t,x) - \varphi(t,x)|}{1 + |x-x_0|^\theta}
    \leq  \sup_{t \in [0,T]} \sup_{|x-x_0|>n} \frac{|\varphi(t,x)|}{1 + |x-x_0|^\theta}.
  \end{align}
  Therefore, because we chose $\varphi \in C_{x_0,\theta}([0,T]\times \mathbb{R}^d)$,
  \[\lim_{n \to \infty} \sup_{t \in [0,T]} \sup_{x \in \mathbb{R}^d} \frac{|\varphi_n(t,x) - \varphi(t,x)|}{1 + |x-x_0|^\theta}=0.\]

  By Lipschitz continuity of $\mathcal{M}$ \eqref{eq:M-Lip-w-proof} and the density of $C_b([0,T]\times \mathbb{R}^d)$ in $C_{x_0,\theta}([0,T]\times \mathbb{R}^d)$ it follows that $\mathcal{M}$ can be uniquely extended to an operator from $C_{x_0,\theta}([0,T]\times \mathbb{R}^d) \to C_{x_0,\theta}([0,T]\times \mathbb{R}^d)$ satisfying \eqref{eq:M-Lip-w-proof}.

  Under assumption \eqref{eq:f-growth} that $|f(u)| \leq K e^{K |u|^\nu}$, we can show that the extension of $\mathcal{M}$ satisfies \eqref{eq:M-def} if $\theta \in (0,2/\nu)$. Let $z \in C_{x_0,\theta}([0,T]\times \mathbb{R}^d)$ and let $z_n$ be a sequence in $C_b([0,T]\times \mathbb{R}^d)$ that converges to $z$ in the $C_{x_0,\theta}([0,T]\times \mathbb{R}^d)$ norm. Let $u_n = \mathcal{M}(z_n)$. Each $u_n$ satisfies
  \begin{equation} \label{eq:un-eq}
    u_n(t,x) = \int_0^t \int_{\mathbb{R}^d} G(t-s,x-y)f(u_n(s,y))dyds + z_n(t,x).
  \end{equation}
  By the bound on $f$, for any $(s,y) \in [0,t]\times\mathbb{R}^d$,
  \begin{equation}
    |f(u_n(s,y))| \leq K e^{K|u_n(s,y)|^\nu}
    \leq K e^{K|u_n|_{C_{x_0,\theta}([0,T]\times \mathbb{R}^d)}^\nu(1 + |y-x_0|^\theta)^{\nu}}.
  \end{equation}
  By \eqref{eq:M-a-priori}, the $|u_n|$ are uniformly bounded in $C_{x_0,\theta}([0,T]\times \mathbb{R}^d)$. There exists $C>0$ such that for any $x,y \in \mathbb{R}^d$ and $0\leq s\leq t$,
  \[\sup_n G(t-s,x-y) |f(u_n(s,y))| \leq \frac{K}{(2\pi(t-s))^{\frac{d}{2}}} e^{-\frac{(x-y)^2}{2(t-s)}} e^{C|y-x_0|^{\theta \nu}}.  \]
  For any fixed $x \in \mathbb{R}^d$ and $t>0$,  when $\theta \nu<2$
  \[\int_0^t \int_{\mathbb{R}^d} \frac{K}{(2\pi(t-s))^{\frac{d}{2}}} e^{-\frac{(x-y)^2}{2(t-s)}+C|y-x_0|^{\theta \nu}}dyds<+\infty\]
  By \eqref{eq:un-eq}, the dominated convergence theorem, and the fact that $u_n(t,x) \to u(t,x)$ and $z_n(t,x) \to z(t,x)$,
  \[u(t,x) = \int_0^t \int_{\mathbb{R}^d} G(t-s,x-y)f(u(s,y))dyds + z(t,x).\]

\end{proof}

Finally, we state without proof that whenever $z(t,x)$ is a random field that is adapted to a filtration $\mathcal{F}_t$,  $\mathcal{M}(z)(t)$ is adapted to the same filtration. This is clear from \eqref{eq:M-def}.
\begin{proposition}
  Let $\theta \in (0, 2/\nu)$. If $z(t,x)$ is a random field such that $$\Pro(z \in C_{x_0,\theta}([0,T]\times \mathbb{R}^d))=1$$ and it is adapted with respect to an increasing right-continuous filtration $\mathcal{F}_t$, then $\mathcal{M}(z)(t,x)$ is a well-defined random field adapted to the same filtration.
\end{proposition}

\section{Proof of Theorem \ref{thm:exist-unique}} \label{S:existence}

%
%
%

\begin{proof}
  If a solution $u(t,x)$ to \eqref{eq:mild} exists, then define
  \[Z(t,x) = \int_0^t \int_{\mathbb{R}^d} G(t-s,x-y)\sigma(u(s,y))W(dyds).\]

  Let $U_0(t,x) = \int_{\mathbb{R}^d}G(t,x-y)u_0(y)dy$. Then letting $\mathcal{M}$ be the mapping defined in \eqref{eq:M-def},
  \[u(t,x) = \mathcal{M}( U_0 + Z)(t,x).\]

  In particular, this means that
  \begin{equation} \label{eq:Z-mild-form}
    Z(t,x) = \int_0^t \int_{\mathbb{R}^d} G(t-s,x-y)\sigma(\mathcal{M}(U_0 + Z)(s,y))W(dyds).
  \end{equation}
  We will show that a solution to \eqref{eq:Z-mild-form} exists via a Picard iteration argument. Then $u = \mathcal{M}(U_0 + Z)$ is a solution to \eqref{eq:mild}.
  Let
  \begin{align}
    &Z_0(t,x) := 0\\
    &Z_{n+1}(t,x) : = \int_0^t \int_{\mathbb{R}^d} G(t-s,x-y)\sigma(\mathcal{M}(U_0 + Z_n)(s,y))W(dyds).
  \end{align}
  Each of these $Z_n$ is well defined and for $\theta, p$ satisfying \eqref{eq:init-data-assumption}-\eqref{eq:p-cond} each $Z_n$ satisfies
  \begin{equation} \label{eq:Z_n-well-posed}
    \sup_{x_0 \in \mathbb{R}^d} \E \left|\sup_{t \in [0,T]} \sup_{x \in \mathbb{R}^d} \frac{|Z_n(t,x)|}{1 + |x-x_0|^\theta} \right|^p<+\infty
  \end{equation}
  by the Lipschitz continuity of $\sigma$ (Assumption \ref{assum:sigma}), the linear growth of $\mathcal{M}$ (Corollary \ref{cor:M-lin-growth}), the properties of $u_0$ (Assumption \ref{assum:init-data}), and the properties of the stochastic convolution (Theorem \ref{thm:stoch-conv-growth-rate}). We omit the details for the proof of \eqref{eq:Z_n-well-posed} because the argument is nearly identical to the following argument proving that $\{Z_n\}$ is a contraction.

  We will show that $\{Z_n\}$ is a contraction in the metric
  \[\sup_{x_0 \in \mathbb{R}^d} \E \left|\sup_{t \in [0,T]}\sup_{x \in \mathbb{R}^d} \frac{|Z_{n+1}(t,x) - Z_n(t,x)|}{1 + |x-x_0|^\theta} \right|^p\]
 when $\theta$ and $p$ satisfy \eqref{eq:init-data-assumption}-\eqref{eq:p-cond} and $T>0$ is small enough.

  Because $p> \frac{(\theta+1)(d+1)}{\theta}$ implies $\theta> \frac{d+1}{p-(d+1)}$, \eqref{eq:stoch-conv-bound} guarantees that there exists $C_T>0$ such that
  \begin{align}
    &\sup_{x_0 \in \mathbb{R}^d} \E \left|\sup_{t \in [0,T]}\sup_{x \in \mathbb{R}^d}\frac{ |Z_{n+1}(t,x) - Z_n(t,x)|}{1 + |x-x_0|^\theta} \right|^p\nonumber\\
    &\leq C_T \sup_{s \in [0,T]} \sup_{x_0 \in \mathbb{R}^d} \E \left|\sigma(\mathcal{M}(U_0 + Z_{n})(s,x_0)) - \sigma(\mathcal{M}(U_0+Z_{n-1})(s,x_0))\right|^p.
  \end{align}
  Applying the Lipschitz continuity of  $\sigma$ (Assumption \ref{assum:sigma}), we see that
  \begin{align*}
    &\sup_{x_0 \in \mathbb{R}^d} \E \left|\sup_{t \in [0,T]}\sup_{x \in \mathbb{R}^d}\frac{ |Z_{n+1}(t,x) - Z_n(t,x)|}{1 + |x-x_0|^\theta} \right|^p\nonumber\\
    &\leq C_T \sup_{s \in [0,T]} \sup_{x_0 \in \mathbb{R}^d}\E \left|\mathcal{M}(U_0+Z_n)(s,x_0) - \mathcal{M}(U_0+ Z_{n-1})(s,x_0)\right|^p.
  \end{align*}
  We can trivially bound the right-hand side of the above expression
  \begin{align*}
    &\sup_{s \in [0,T]} \sup_{x_0 \in \mathbb{R}^d} \E\left|\mathcal{M}(U_0+Z_n)(s,x_0) - \mathcal{M}(U_0+ Z_{n-1})(s,x_0)\right|^p \\ &\leq \sup_{x_0 \in \mathbb{R}^d} \E \left|\sup_{s \in [0,T]} \sup_{x \in \mathbb{R}^d} \frac{|\mathcal{M}(U_0 + Z_n)(s,x) - \mathcal{M}(U_0 + Z_{n-1})(s,x)}{1 + |x-x_0|^\theta} \right|^p.
  \end{align*}
  By the Lipschitz continuity of $\mathcal{M}$ (see \eqref{eq:M-Lip-w-proof}) in the $C_{x_0,\theta}([0,T]\times \mathbb{R}^d)$ norm (remembering that the Lipschitz constant does not depend on $x_0$),
  \begin{align}
    &\sup_{x_0 \in \mathbb{R}^d} \E \left|\sup_{t \in [0,T]}\sup_{x \in \mathbb{R}^d} \frac{|Z_{n+1}(t,x) - Z_n(t,x)|}{1 + |x-x_0|^\theta} \right|^p\nonumber\\
    &\leq C_T e^{CT p} \sup_{x_0 \in \mathbb{R}^d} \E \left|\sup_{s \in [0,T]}\sup_{x \in \mathbb{R}^d} \frac{|Z_n(s,x) - Z_{n-1}(s,x)|}{1 + |x-x_0|^\theta} \right|^p.
  \end{align}

  By choosing $T_0$ small enough so \eqref{eq:C-small} guarantees that $C_{T_0} e^{CT_0p}< 1$ we have a contraction. Therefore, $Z_n$ converges to a limit $Z$. $Z$ must be a solution to \eqref{eq:Z-mild-form} on $[0,T_0]$ and
  \begin{equation*} 
    \sup_{x_0 \in \mathbb{R}^d}\E \left|\sup_{t \in [0,T_0]}\sup_{x \in \mathbb{R}^d} \frac{|Z(t,x)|}{1 + |x-x_0|^\theta} \right|^p<+\infty.
  \end{equation*}
  Then $u = \mathcal{M}(U_0 + Z)$ is a mild solution to \eqref{eq:mild} and by Corollary \ref{cor:M-lin-growth},
  \[\sup_{x_0 \in \mathbb{R}^d}\E \left|\sup_{t \in [0,T_0]}\sup_{x \in \mathbb{R}^d} \frac{|u(t,x)|}{1 + |x-x_0|^\theta} \right|^p<+\infty.\]
  We can repeat the procedure with $u_0(x) = u(T_0,x)$ (notice that $u(T_0,\cdot)$ satisfies Assumption \ref{assum:init-data}) to get a solution for $u$ on $[T_0,2T_0]$. By repeating this procedure, we can show that a global in time solution $Z$ solving \eqref{eq:Z-mild-form} exists.

  These arguments also show that for any $T>0$ and $p$ and $\theta$ satisfying \eqref{eq:init-data-assumption}-\eqref{eq:p-cond},
  \begin{equation} \label{eq:Z-bound}
    \sup_{x_0 \in \mathbb{R}^d}\E \left|\sup_{t \in [0,T]}\sup_{x \in \mathbb{R}^d} \frac{|Z(t,x)|}{1 + |x-x_0|^\theta} \right|^p<+\infty.
  \end{equation}
  By Corollary \ref{cor:M-lin-growth} and \eqref{eq:Z-bound}, the solution $u = \mathcal{M}(U_0 + Z)$ satisfies \eqref{eq:solution-bounds}.

  Now we prove uniqueness. The Picard iteration arguments prove that the solution $Z$ to \eqref{eq:Z-mild-form} is unique in the class of adapted random fields satisfying \eqref{eq:Z-bound}. If $u_2$ is another mild solution to \eqref{eq:mild} that satisfies \eqref{eq:Lp-bounded}, let
  \[Z_2(t,x):= \int_0^t \int_{\mathbb{R}^d} G(t-s,x-y)\sigma(u_2(s,y))W(dyds). \]
  Note that $u_2 = \mathcal{M}(U_0 + Z_2)$ and therefore, $Z_2$ solves \eqref{eq:Z-mild-form}.
  Then by \eqref{eq:stoch-conv-bound} and \eqref{eq:Lp-bounded}, $Z_2$ satisfies \eqref{eq:Z-bound}. Therefore, by our previous Picard iteration arguments, $Z_2 = Z$ and by \eqref{eq:M-Lip-w-proof}, $u_2 =u$.
\end{proof}

\section*{Acknowledgements}
The author thanks Le Chen for several interesting discussions.

\bibliography{SPDE}
\bibliographystyle{amsalpha}

\end{document}